\documentclass[12pt]{amsart}
\usepackage{amscd,amssymb,graphics,color,a4wide,hyperref,verbatim}

\usepackage{multicol, fullpage} 
\usepackage{tikz-cd}
\usetikzlibrary[positioning,patterns] 
\usetikzlibrary{matrix,arrows,decorations.pathmorphing}

\usepackage{graphicx}
\usepackage{psfrag}
\usepackage{marvosym}
\usepackage{amsmath}
\usepackage{amsfonts}
\usepackage{amssymb}
\usepackage{amsthm}
\usepackage{mathrsfs}
\usepackage{enumitem}

\usepackage{ytableau}

\input xy
\xyoption{all}

\usepackage{calligra}
\usepackage{mathrsfs}
\DeclareMathAlphabet{\mathcalligra}{T1}{calligra}{m}{n}
\DeclareFontShape{T1}{calligra}{m}{n}{<->s*[1.5]callig15}{}

\newtheorem{theorem}{Theorem}[section]
\newtheorem*{theoremstar}{Theorem}
\newtheorem{lemma}[theorem]{Lemma}

\newtheorem{corollary}[theorem]{Corollary}

\theoremstyle{definition}
\newtheorem{definition}[theorem]{Definition}

\newtheorem{example}[theorem]{Example}

\newtheorem{remark}[theorem]{Remark}
\newtheorem{theorem-definition}[theorem]{Theorem-Definition}

\numberwithin{equation}{section}

\renewcommand{\AA} {\mathbb{A}}

\newcommand{\CC} {\mathbb{C}}
\newcommand{\DD} {\mathbb{D}}

\newcommand{\GG} {\mathbb{G}}

\newcommand{\LL} {\mathbb{L}}

\newcommand{\PP} {\mathbb{P}}

\newcommand{\RR} {\mathbb{R}}

\newcommand{\ZZ} {\mathbb{Z}}

\newcommand {\shA} {\mathcal{A}}
\newcommand {\shB} {\mathcal{B}}
\newcommand {\shC} {\mathcal{C}}

\newcommand {\shE} {\mathcal{E}}
\newcommand {\shF} {\mathcal{F}}

\newcommand {\shH} {\mathcal{H}}

\newcommand {\shK} {\mathcal{K}}

\newcommand {\shM} {\mathcal{M}}

\newcommand {\shT} {\mathcal{T}}
\newcommand {\shP} {\mathcal{P}}

\newcommand {\sC} {\mathscr{C}}
\newcommand {\sD} {\mathscr{D}}
\newcommand {\sE} {\mathscr{E}}
\newcommand {\sF} {\mathscr{F}}
\newcommand {\sG} {\mathscr{G}}

\newcommand {\sI} {\mathscr{I}}

\newcommand {\sL} {\mathscr{L}}

\newcommand {\sN} {\mathscr{N}}
\newcommand {\sO} {\mathscr{O}}

\newcommand {\sU} {\mathscr{U}}

\newcommand {\fol}  {\mathfrak{l}}

\newcommand {\fos}  {\mathfrak{s}}

\newcommand {\foR} {\mathfrak{R}}

\newcommand{\blank}{\underline{\hphantom{A}}}

\newcommand {\codim} {\operatorname{codim}}

\newcommand {\coh} {\operatorname{coh}}
\newcommand {\Coker} {\operatorname{Coker}}

\newcommand {\cone} {\operatorname{cone}}

\newcommand {\D} {\operatorname{D}}

\newcommand {\End} {\operatorname{End}}

\newcommand{\sExt}{\mathscr{E} \kern -1pt xt}

\newcommand{\Hilb}{\mathrm{Hilb}}

\newcommand {\Hom} {\operatorname{Hom}}
\newcommand {\sHom}{\mathscr{H}\kern-5pt\mathcalligra{om}}

\newcommand {\id} {\operatorname{id}}
\newcommand {\Id} {\operatorname{Id}}
\newcommand {\im} {\operatorname{im}}

\newcommand {\kk} {\Bbbk}
\renewcommand {\ker } {\operatorname{Ker}}
\newcommand {\Ker} {\operatorname{Ker}}

\newcommand {\Pic} {\operatorname{Pic}}

\newcommand {\Proj} {\operatorname{Proj}}

\newcommand {\pr} {\operatorname{pr}}

\newcommand {\rank} {\operatorname{rank}}

\newcommand {\Spec} {\operatorname{Spec}}

\newcommand {\Sym} {\operatorname{Sym}}

\newcommand{\sTor}{\mathscr{T} \kern -3pt or}
\newcommand {\Tot} {\operatorname{Tot}}

\newcommand {\Bl} {\operatorname{Bl}}

\newcommand {\tX} {X_\sigma^-}
\newcommand {\pX} {X_\sigma^+}

\newcommand {\bL} {\mathbf{L}}
\newcommand {\bR} {\mathbf{R}}

\title[]{Derived categories of projectivizations and flops}
\author[Q.Y.\ JIANG, N.C.\ Leung]{Qingyuan Jiang, Naichung Conan Leung}

\address{School of Mathematics, University of Edinburgh, James Clerk Maxwell Building, Peter Guthrie Tait Road, Edinburgh EH9 3FD, United Kingdom.}
\email{qingyuan.jiang@ed.ac.uk}

\address{The Institute of Mathematical Sciences and Department of Mathematics,
The Chinese University of Hong Kong, Shatin, N.T., Hong Kong}\email{leung@math.cuhk.edu.hk}

\begin{document}

\begin{abstract} We prove a generalization of Orlov's projectivization formula for the derived category $D^b_{\rm coh} (\mathbb{P}(\mathscr{E}))$, where $\mathscr{E}$ does not need to be a vector bundle; Instead, $\sE$ is a coherent sheaf which locally admits two-step resolutions. As a special case, this also gives Orlov's generalized universal hyperplane section formula.  As applications, (i) we obtain a blowup formula for blowup along codimension two Cohen-Macaulay subschemes, (ii) we obtain new ``flop-flop=twist" results for a large class of flops obtained by crepant resolutions of degeneracy loci.  As another consequence, this gives a perverse schober on C. (iii) we give applications of above results to symmetric powers of curves and $\Theta$-flops, following Toda \cite{Tod2}. 

\vspace{-2mm} 
\end{abstract}

\maketitle

\section{Introduction}

The derived category of coherent sheaves $D(X) := D^b_{\coh}(X)$ on a scheme $X$, introduced by Grothendieck and Verdier in the 1950s, is a primary algebraic invariant for $X$. An important question is that how derived categories behave under basic geometric operations. 

\subsection{Projectivization formula}
In \cite{Orlov92}, Orlov shows that the derived category of a {projective bundle} $\PP(\sG)$ over $X$ consists exactly of $r$ copies of $D(X)$, where $\sG$ is a vector bundle of rank $r$, see Thm. \ref{thm:proj_bundle}. It is an interesting question what happens if $\sG$ is not locally free. 

In this paper, we answer this question in the case when $\sG$ is a coherent sheaf of  homological dimension $\le 1$ on a regular scheme $X$, i.e. Zariski locally over $X$, $\sG$ admits a two-step resolution 
	$0 \to \sF \xrightarrow{\sigma} \sE \to \sG \to 0$,
where $\sF$ and $\sE$ are locally free of rank $f$ and $e$ respectively. Then $\sG$ has rank $r: = e - f \ge 0$.
Notice that this condition is equivalent to $\sExt^i_{\sO_X}(\sG,\sO) =0$ for $i \ge 2$ (if $X$ is regular), and it is satisfied by $\sG = {\rm coker}(\sigma)$ for a general $\sO_X$-module morphism $\sigma \colon \sF \to \sE$, if $\sHom_X(\sF,\sE)$ is globally generated. 

$\PP(\sG)$ is generically a $\PP^{r-1}$-bundle over $X$, and a generic $\PP^r$-bundle over the {\em degeneracy locus} $X_{\sigma} = \{x \in X \mid \rank \sG(x) > r \}$; The fiber dimension jumps by more over further degeneracy loci, see \S \ref{sec:deg}. On the other hand, the sheaf $\sExt^1(\sG,\sO_X)$ is supported on $X_{\sigma}$, and $\PP(\sExt^1(\sG,\sO_X))$ is a {Springer--type partial desingularization} of $X_\sigma$, locally given by:
	$$\PP(\sExt^1(\sG,\sO_X)) |_U  = \{(x, H_x) \mid \im \sigma^\vee(x) \subset H_x \} \subset \PP_U(\sF^\vee).$$
where $x \in U \subset X$, and $H_x \subset \sF^\vee(x)$ denotes a hyperplane. Our first main result is:

\begin{theoremstar}[Projectivization formula, Thm. \ref{cor:projectivization}] There is a semiorthogonal decomposition 
	$$D(\PP(\sG)) = \big \langle D(\PP(\sExt^1(\sG,\sO_X))), ~D(X)(1), \ldots, D(X)(r)\big \rangle,$$
provided that $\PP(\sG)$ and $\PP(\sExt^1(\sG,\sO_X))$ are of the expected dimensions.
\end{theoremstar}

This theorem simultaneously generalizes (i) Orlov's projective bundle formula  \cite{Orlov92} (corresponding to the case $\sExt^1(\sG,\sO_X) = 0$), see Thm. \ref{thm:proj_bundle}; (ii) Orlov's generalized universal hyperplane section formula  \cite{Orlov05} (corresponding to the case $\sExt^1(\sG,\sO_X) \simeq \sO_Z$, where $Z \subset X$ a local complete intersection subscheme), see Thm. \ref{thm:HPDI}; and (iii) the derived equivalence for flops obtained by two Springer partial desingularizations of determinantal hypersurfaces (corresponding to the case $r = 0$), see \S \ref {sec:springer}. 

If $\rank \sG=1$, this formula implies a {\em blowup formula} for blowup along a {\em codimension two Cohen-Macaulay subscheme $Z \subset X$}, see \S \ref{sec:CM2}. We will discuss more applications below.

Orlov's result Thm. \ref{thm:HPDI} was used in \cite{KKLL} to show that all complete intersections in $\PP^n$ are Fano visitors. Our projectivization formula should have applications to the problem of Fano visitors for non-complete intersections subschemes and for resolutions of degeneracy loci. 

Our proof is based on the  ``chess game", introduced by Richard Thomas in his reinterpretation \cite{RT15HPD} of Kuznetsov's work \cite{Kuz07HPD}, and further developed in \cite{JLX17}. We will only need a simple situation of the general chess game, the ``rectangular case"; the results needed in this situation are reviewed in Appendix \ref{sec:app}.  The chess game method can be applied to many situations, including cases of homological projective duality and various flops, see \cite{JLX17, RT15HPD, JL18join, JL18Bl}. 

\subsection{``Flop--flop=twist" results}

The ``flop--flop=twist" phenomenon was famously known for Atiyah flops $X_1 \xleftarrow{q_1} \widetilde{X} \xrightarrow{q_2} X_2$, where $q_1, q_2$ are blowing up of along $(-1,-1)$-curves inside threefolds $X_1$ and $X_2$ respectively. Bondal and Orlov \cite{BO} show that the flopping functor
	$$ \RR q_{2*} \, \LL q_1^{*}: D(X_1) \to D(X_2)$$
	is an equivalence of categories. In this case, the composition
		$$D(X_1) \xrightarrow{\RR q_{2*} \, \LL q_1^{*}} D(X_2) \xrightarrow{\RR q_{1*} \, \LL q_2^{*}} D(X_1)$$
	 of ``flopping equivalences" is not identity, but an interesting nontrivial autoequivalence of $D(X_1)$: the {\em spherical twist} around {\em spherical object} $\sO_{\PP^1}(-1) \in D(X_1)$, introduced by Seidel and Thomas \cite{ST01}. The concept of spherical twists around spherical objects is generalized by Anno \cite{Ann}, Anno and Logvinenko \cite{AL16}, and Rouquier \cite{Rou} to spherical twists around spherical functors, and by Huybrechts and Thomas to $\PP$-twists around $\PP$-objects \cite{HT}.

The ``flop--flop=twist" result for Atiyah flops inspires the study of similar results for other type of flops. For general flops of smooth threefolds, the derived equivalences are established by Bridgeland \cite{Bri}, and later generalized by Chen \cite{Chen} to threefolds with terminal Gorenstein singularities. Toda \cite{Tod} shows the ``flop--flop=twist" phenomena for a class of threefold flops, where the twists are around fat spherical objects. Bodzenta and Bondal \cite{BB} shows similar phenomena for flops of relative dimension $1$. Donovan and Wemyss \cite{DW1, DW2} show it for a class of threefold flops, and the spherical functors are given over certain non-commutative algebras. Cautis \cite{Cau},  Addington, Donovan and Meachan \cite{ADM} show the results for Mukai flops, whose derived equivalences are established by Kawamata \cite{Kaw02} and Namikawa \cite{Nam03}. Hara \cite{Hara17} studies the Abuaf flops whose derived equivalences are obtained by Ed Segal \cite{Seg}. Donovan and Segal show similar results for Grassmannian flops \cite{DS}.

We show that another large class of flops also fit into this picture. More precisely, assume $\sG = {\rm coker} (\sF \xrightarrow{\sigma} \sE)$ and $\rank \sG = 0$. Assume further that $X_1: = \PP(\sG)$ and $X_2: = \PP(\sExt^1(\sG,\sO))$ are smooth and of expected dimensions 
	$$\dim X_1  = \dim X_2 = \dim X_1 \times_X X_2 = \dim X -1.$$ 
Therefore the birational map $X_1 \dashrightarrow X_2$ from the diagram
	$$X_1  \xleftarrow{q_1} X_1 \times_X X_2 \xrightarrow{q_2} X_2$$
 is a flop obtained by two (Springer-type) crepant partial desingularizations of degeneracy locus $X_\sigma$ (see \S \ref{sec:springer}). Denote $\pi_i: X_i \to X$ the natural projections, $i = 1,2$. 

\begin{theoremstar}[``Flop--flop=twist", see Thm. \ref{thm:flop}]  For each $k \in \ZZ$ the functor
\begin{equation*} \Phi_k: D(X_1) \longrightarrow D(X_2), \qquad \Phi_k(\blank) = \RR q_{2*} \, (\LL q_1^{*}(\blank) \otimes \sO(k,k)),\end{equation*}
is an equivalence of categories, and the equivalence functors $\Phi_{k} \colon D(X_1) \simeq D(X_2)$ are related by spherical twists. In particular, the ``flop--flop" equivalence functor:
	$$D(X_1) \xrightarrow{\RR q_{2*} \, \LL q_1^{*}} D(X_2) \xrightarrow{\RR q_{1*} \, \LL q_2^{*}} D(X_1)$$
is a spherical twist (up to twists by line bundles and degree shifts), see Thm. \ref{thm:flop}. 
\end{theoremstar}

 As a consequence, in the case of an Atiyah flop $X_+ \dasharrow X_-$ arising from  two small crepant resolutions of the determinantal threefold $X_\sigma \subset X = \CC^4$, we obtain the following:
	$$T_{\LL \pi^{+*}} = T_{\sO_{\PP^1}(-1)} [2] \in {\rm Autoeq}(D(X_+)),$$
 where $\pi^{+}: \pX \to X$ the projection, and $T_{\LL \pi^{+*}}$ (resp. $T_{\sO_{\PP^1}(-1)}$) is the twist functor around the spherical functor $\LL \pi^{+*}$ (resp. spherical object $\sO_{\PP^1}(-1)$); see Ex. \ref{ex:Atiyah}. See also \S \ref{sec:app:Atiyah}, Lem. \ref{lem:app:Atiyah} for a direct proof of this equality.
 
As another consequence, our theorem also shows that the diagram of categories
  \begin{equation*}
      \xymatrix{
D(X_1) \ar@<-.7ex>[r]_{\LL q^*_1}& D(X_0)
\ar@<-.7ex>[l]_{\RR q_{1*}}
\ar@<.7ex>[r]^{\RR q_{2*}}& D(X_2)
\ar@<.7ex>[l]^{\LL q^*_2}   
}
\end{equation*} 
represents a {\em perverse schober} on $\CC$, where $X_0=X_1 \times_X X_2$.
Perverse schobers are the conjectured categorification of perverse sheaves, proposed by Kapranov and Schechtman \cite{KS}. Bondal, Kapranov and Schechtman \cite{BKS} suggests a close connection between perverse schobers and  flops. The cases of flops of relative dimension one has been studied in \cite{Don, BKS, BB}; \cite{BKS} also studies the cases of Grothendieck resolutions. Our results provide a sequence of examples of higher relative dimensional flops that also present perverse schobers over a disk, and provide further evidences for the proposal of Bondal, Kapranov and Schechtman.
  
It should be noted that our spherical functors are global (similar to \cite{BB}) in nature, see Lem. \ref{lem:spherical_stratified}. Moreover, our spherical functors are naturally stratified. It would be interesting to compare our functor with Horja's EZ-functors \cite{Hor}, Toda's spherical functors in \cite{Tod}, and the twist functors in the noncommutative geometry of Donovan and Wemyss \cite{DW1,DW2}. 


\subsection{Applications to symmetric powers of curves and $\Theta$--flops} \label{sec:intro:SymC} The following applications were brought to the authors' attention by Kapranov; see also Toda's work \cite{Tod2}. Let $C$ be a smooth projective curve of genus $g \ge 1$. 
Then the {\em projectivization formula} above reproduces Toda's formula \cite[Cor. 5.11]{Tod2} on symmetric powers of curves: for all $n \ge 0$, 
	\begin{align*}
	D(\Sym^{g-1+n} (C))   = \big\langle D(\Sym^{g-1-n} (C)), ~ 
	 D(\Pic^{g-1+n}(C))(1), \ldots, D(\Pic^{g-1+n}(C))(n)\big\rangle.
	\end{align*}
See Cor. \ref{cor:SymC}; Here, $\Sym^k(C)$ denotes the $k$-th symmetric power of $C$. This application has also been independently discovered by Belmans--Krug \cite{BK19}. 

In the case $n=0$, this yields the derived equivalence for the {\em $\Theta$-flop}:
	$$X_+ =\Sym^{g-1} C \dashrightarrow X_- =\Sym^{g-1} C,$$ 
where $X_{\pm} \to \Theta$ give two different resolutions of the theta divisor $\Theta \subset \Pic^{g-1}(C)$. Moreover, in this case, we obtain a sequence of derived equivalences $\Phi_k \colon D(X_+) \simeq D(X_-)$ for all $k \in \ZZ$, and show that these equivalences are connected with each other by spherical twists. In particular, the derived pullback functor $\LL (AJ)^*$ of the Abel-Jacobi map 
	$$AJ \colon C^{(g-1)}: = \Sym^{g-1}(C) \to \Pic^{g-1}(C)$$
is {\em spherical}, and the {\em ``flop--flop=twists"} result holds for the {$\Theta$-flops}. See \S \ref{sec:Theta}, Cor. \ref{cor:Theta}.

Dimensions and smoothness of the fibre products $\Sym^{g-1+n}(C) \times_{\Pic^{g-1+n}(C)} \Sym^{g-1-n}(C)$ are discussed in Appendix \ref{sec:app:SymC}. This answers a question posed by M. Kapranov.

\subsection{Other applications} In Belmans and Krug's paper \cite{BK19}, our theorem, in particular the blowup formula in \S \ref{sec:CM2}, is applied to study the derived category of nested Hilbert schemes of points on surfaces. We include this application in \S\ref{sec:Hilb} for the reader's convenience, and provide more details about the Fourier-Mukai kernels. 

 In \S \ref{sec:Voisin}, we apply our result to {\em Voisin maps}  \cite{Voi16}, \cite{Chen}, which are rational maps from the product of Fano variety of lines on a cubic fourfold to the associated LLSvS eightfold \cite{LLSVS17}.
 
\subsection{Related works}
The projectivization formula is essentially known, at least in local situations and in a different formulation, to Kuznetsov \cite{Kuz07HPD}. The flop case of the projectivization formula is related to the works of Buchweitz, Leuschke and Van den Bergh \cite{BLV}, Weyman and Zhao \cite{WZ} and Donovan and Segal \cite{DS}. Our approach is based on the ``chess game" method (see \cite{JLX17} and Appendix \ref{sec:app}), which allows us to give explicit descriptions for the Fourier-Mukai functors and to obtain ``flops--flops=twists" in the case of flops. It would be interesting to find the connections of this work with noncommutative resolutions considered by Buchweitz, Leuschke and Van den Bergh \cite{BLV}, Weyman and Zhao \cite{WZ}.

The projectivization formula turns out to be closely related to Toda's results on SODs from d-critical flips \cite{Tod3,Tod2}. In fact the relationship between $\PP(\sG)$ and $\PP(\sExt^1(\sG,\sO))$ is analytically locally a {\em d-critical flip}, and the techniques developed in \cite{Tod3,Tod2} could be applied. In turn, our results on projectivization formula and spherical twists could be applied to the case of symmetric powers of curves considered in \cite{Tod2}, see \S \ref{sec:SymC} and \S \ref{sec:Theta}.

Our projectivization formula Thm. \ref{thm:duality} compares nicely with Pirozhkolv's theorem \cite{Pir}. In fact, a special case of our Thm. \ref{thm:duality} (if $f=1$) agrees with a special case of Pirozhkolv's theorem (if $d=1$); These two results are in general complementary. (Both results fit into the general framework of Quot schemes of locally free quotients; see \cite{J20, J21}.) 

A Chow-theoretical version of the projectivization formula has been established in \cite{J19}.

\subsection{Conventions} We will use {\bf Grothendieck's notations}: for a coherent sheaf $\sF$ on a scheme $X$, denote by $\PP_X(\sF) = \PP(\sF): =  \Proj_X \Sym_{\sO_X}^\bullet \sF$ its projectivization, and by $C(\sF): = \Spec \Sym^\bullet \sF$ the abelian cone scheme. Notice that if $\sF$ is locally free, then its underlying vector bundle is $|\sF| = C(\sF^\vee)$, and the space $\PP_{\rm sub}(\sF) :=\PP(\sF^\vee)$ parametrizes one-dimensional sub-bundles of $\sF$.

Let $\kk$ be a fixed field throughout this paper. All schemes are assumed to be  noetherian, separated $\kk$-schemes of finite Krull dimensions. (Hence by a {\em regular} scheme, we always mean a connected, regular, noetherian, separated $\kk$-scheme of finite Krull dimension.) For a scheme $X$, $D(X) : = D^b(\coh X)$ denotes its  {\em bounded} derived category of coherent sheaves, and $D_{\rm qc}(X)$ denotes its unbounded derived category of complexes of $\sO_X$-modules with quasi-coherent cohomologies. Since we assume that $X$ is noetherian, by \cite[Cor. II. 2.2.2.1]{SGA6}, $D(X)$ is naturally identified with the subcategory $D^b_{\coh}(X) \subseteq D_{\rm qc}(X)$ of complexes with bounded, coherent cohomologies. 

Let $f \colon X \to Y$ be a proper, local complete intersection morphism (i.e. locally, $f$ factorizes as a composition of a regular closed immersion followed by a smooth morphism); This condition will be satisfied by all the morphisms considered in this paper. Then the derived pullback and derived pushforward are well-defined on the level of the bounded derived category of coherent sheaves: $\LL f^* \colon D(Y) \to D(X)$ and $\RR f_* \colon D(X) \to D(Y)$. In fact, $\LL f^*$ always preserves (pseudo)-coherence, and it is bounded since $f$ has finite Tor-dimension; $\RR f_*$ is bounded since $f$ is proper, it preserves coherence since $f$ is proper and pseudo-coherent, see \cite[Cor. III. 2.5]{SGA6}.  Furthermore, $\LL f^* \colon D(Y) \to D(X)$ admits a left adjoint $f_{!} \colon D(X) \to D(Y)$, and $\RR f_* \colon D(X) \to D(Y)$ admits a right adjoint $f^{!} \colon D(Y) \to D(X)$. If we denote $\omega_f^\bullet = f^!(\sO_Y) $ the relative dualizing complex of $f$, then $\omega_f^\bullet= \omega_f [\dim f]$ is a shift of line bundle, and the following holds: $f^!(A) = f^*(A) \otimes \omega_f^\bullet$, and $f_!(A) = f_*(A \otimes \omega_f^\bullet)$.

\subsection*{Acknowledgement} 
J.Q. is grateful for Richard Thomas, Ed Segal, Zak Turcinovic and Francesca Carocci for many inspiring discussions on derived categories, Cayley's trick and projectivization, for J{\'a}nos Koll{\'a}r, Clemens Koppensteiner, Jakub Witaszek and Ziquan Zhuang for discussions on flops. Both authors would like to thank Mikhail Kapranov, Andrei C{\u{a}}ld{\u{a}}raru, Matthew Young, Ying Xie, and Dan Wang for helpful discussions.  

We especially thank Mikhail Kapranov for bringing to our attention the problem regarding theta divisor and for many helpful discussions; thank Y. Toda for helpful discussions related to his work \cite{Tod2}, and thank Y. Kawamata, A. C{\u{a}}ld{\u{a}}raru, W. Donovan for their interests in this work and helpful conversations during a conference held in SYSU, Guangzhou in May 2019, for which conference the authors thank SYSU and Changzheng Li for hospitality. 

We also thank the referees for their careful reading, many useful suggestions which significantly improve the exposition of the paper.

J.Q. is supported by the Engineering and Physical Sciences Research Council [EP/R034826/1]; L.N.C. is supported by a grant from the Research Grants Council of the Hong Kong Special Administrative Region, China (Project No. CUHK --14301117). During the time the paper was written, J.Q. was also supported by a grant from National Science Foundation (Grant No. DMS -- 1638352) and the Shiing-Shen Chern Membership Fund of IAS.

\section{Preparation: Cayley's trick and degeneracy loci}

\subsection{Semiorthogonal decompositions} We refer the readers to \cite{Huy, Cal, Kuz14sod} for basic notations and properties of derived categories. 

A {\em semiorthogonal decomposition} of a triangulated category $\shT$, denoted by
	$$\shT = \langle \shA_1, \shA_2, \ldots, \shA_n \rangle,$$
 is a sequence of full triangulated subcategories $\shA_1, \shA_2, \ldots, \shA_{n}$ of $\shT$, such that 
 	\begin{enumerate}
		\item $\Hom (A_j ,A_i) = 0$ for all $A_i \in \shA_i$ and $A_j \in \shA_j$ , if $j > i$;
		\item For any $T \in \shT$, there is a sequence of maps $0 = T_n \to T_{n-1} \to \ldots \to T_1 \to T_0 =T$, such that $\cone (T_{i} \to T_{i-1}) \in \shA_i$.
	\end{enumerate}
The subcategories $\shA_i$'s are called {\em components} of $\shT$.

An important example is provided by Orlov's formula for projective bundles:

\begin{theorem}[Projective bundle formula, Orlov \cite{Orlov05}] \label{thm:proj_bundle}
Let $X$ be a regular scheme, and let $\sE$ be a vector bundle of rank $r$ on $X$. Denote $\pi: \PP(\sE) \to X$ the natural projection. Then the functors $\LL \pi^*(-) \otimes \sO_{\PP(\sE)}(k) : D(X) \to D(\PP(\sE))$ are fully faithful, $k = 0, 1, \ldots, r-1$, and their image give rises to a semiorthogonal decomposition
	$$D(\PP(\sE)) = \langle \pi^* D(X), \pi^* D(X) \otimes \sO_{\PP(\sE)}(1), \ldots, \pi^* D(X) \otimes \sO_{\PP(\sE)}(r-1)\rangle.$$ 
\end{theorem}
If $X = \Spec \kk$, this recovers the {\em Beilinson's decomposition} \cite{Bei} for projective spaces.

\subsection{Mutations and spherical twist} 
\subsubsection{Mutations.} \label{sec:mutations} A full triangulated subcategory $\shA$ of a triangulated category $\shT$ is called  {\em left (resp. right) admissible} if the inclusion functor $i = i_{\shA}: \shA\to \shT$ admits a left adjoint $i^*: \shT \to \shA$ (resp. a right adjoint $i^!: \shT \to \shA$); $\shA \subset \shT$ is called {\em admissible} if it is both left and right admissible. If $\shA\subset \shT$ is admissible and $\shT$ is saturated in the sense of \cite{B}, then $\shA^\perp = \{ T \in \shT \mid \Hom(\shA,T) = 0\}$ and ${}^\perp \shA =\{ T \in \shT \mid \Hom(T, \shA) = 0\}$ are both admissible, and there are semiorthogonal decompositions $\shT = \langle \shA^\perp, \shA \rangle =  \langle \shA, {}^\perp \shA\rangle$. Let $\shA \subset \shT$ be an admissible subcategory. Then the functor $\bL_\shA: = i_{\shA^\perp} i^*_{\shA^{\perp}}$  (resp. $\bR_{\shA} : =  i_{{}^\perp \shA} i^!_{ {}^\perp \shA}$) is called the {\em left (resp. right) mutation through $\shA$}.  For any $b \in \shT$, there are exact triangles
		$$ i_\shA i^!_{\shA} (b) \to b \to \bL_{\shA} b \xrightarrow{[1]}{},\qquad  \bR_{\shA} b \to b \to  i_\shA i^*_{\shA} (b) \xrightarrow{[1]}{}.$$
Furthermore, $(\bL_{\shA})\,|_{\shA} = 0$ and $(\bR_{\shA})\,|_{\shA} = 0$; $(\bL_{\shA})\,|_{{}^\perp \shA} : {}^\perp \shA \to \shA^\perp$ and $(\bR_{\shA})\,|_{\shA^\perp } : \shA^\perp \to {}^\perp \shA $ are mutually inverse equivalences of categories. Staring with a semiorthogonal decomposition $\shT = \langle \shA_1, \ldots, \shA_{k-1}, \shA_k, \shA_{k+1}, \ldots, \shA_n \rangle$ with admissible components, for $k \in [1,n]$, the left and right mutations produce semiorthogonal decompositions:
		\begin{align*}
		\shT&  = \langle \shA_1, \ldots,\shA_{k-2}, \bL_{\shA_{k-1}} (\shA_k), \shA_{k-1}, \shA_{k+1}, \ldots, \shA_n \rangle\\
		& =  \langle \shA_1, \ldots, \shA_{k-1}, \shA_{k+1}, \bR_{\shA_{k+1}} (\shA_k), \shA_{k+2}, \ldots, \shA_n \rangle.
		\end{align*}
We refer the readers to \cite{B, BK, Kuz07HPD} for more details about mutations.

\subsubsection{Spherical twists.} 

Spherical twists introduced by Paul Seidel and Richard Thomas \cite{ST01}, is a very important type of autoequivalences of derived categories of algebraic varieties. The concept was later generalized by Anno \cite{Ann}, Anno--Logvinenko \cite{AL16}  and Rouquier \cite{Rou} to spherical twists around spherical functors, by Huybrechts--Thomas to $\PP$-twist around $\PP$-object \cite{HT}. See \cite{Add} for a nice introduction, and \cite{Tod, Hor, AL16, Kuz15, Seg17, Mea} for various aspects.
  
In this paper, we only consider categories, and the functors between categories, in the $2$-category of Fourier-Mukai kernels of \cite{CW}. Hence it makes sense to talk about taking cones of functors. Suppose $F: \shA \to \shB$ is an exact functor between two triangulated categories (in the $2$-category of Fourier-Mukai kernels), and assume that $F$ admits a left adjoint $L$ and a right adjoint  $R$.
Then the {\em twist} functor $T = T_F $, {\em cotwist} functor $C$, {\em dual twist} functor $T'$ and {\em dual cotwist} functor $C' $ are defined respectively via the following exact triangles:
		\begin{align*}
		&FR \to 1_\shB \to T \xrightarrow{[1]} FR[1], 
		& C\to  1_\shA \to RF \xrightarrow{[1]} C[1], \\
		& T'\to  1_\shB \to FL \xrightarrow{[1]} T'[1],
		& LF \to 1_\shA \to C' \xrightarrow{[1]} LF[1].
		\end{align*}
The mutation functors $\bL_{\shA}$ and $\bR_{\shA}$, defined in \S \ref{sec:mutations}, are special cases of the above-defined functors: $\bL_{\shA}$ (resp. $\bR_{\shA}$) is the twist functor (resp. the dual twist functor) of the inclusion functor $i_{\shA}: \shA \hookrightarrow \shT$. This observation will play an important role in the proof of Thm. \ref{thm:flop}. 

\begin{definition}[Anno \cite{Ann}, Rouqier \cite{Rou}, Anno--Logvinenko \cite{AL16}] \label{def:spherical} $F: \shA \to \shB$ is called {\em spherical}, if the cotwist functor $C$ of $F$ is an autoequivalence of $\shA$, and $R \simeq CL[1]$ holds.
\end{definition}

\begin{theorem}[Anno \cite{Ann}, Rouqier \cite{Rou},  Anno--Logvinenko \cite{AL16}] If $F: \shA \to \shB$ is a spherical functor, then the twist functor $T_F$ of $F$ is an autoequivalence of $\shB$.
\end{theorem}

Here, we follow \cite{Add, Mea}, and only require the existence of an isomorphism $R \simeq CL[1]$. Then it was shown in above references that the above theorem holds under this weaker condition, and furthermore the natural composition $R \to RFL \to CL[1]$ is an isomorphism. 

\subsection{Cayley's trick and Orlov's results} \label{sec:semi-local}
Cayley's trick is a method to relate the geometry of the zero scheme of a regular section of a vector bundle to the geometry of a certain hypersurface. The exposition here mostly follows Thomas \cite{RT15HPD} and Orlov \cite{Orlov05}.

Let $\sE$ be a locally free sheaf of rank $r$ on a regular scheme
$X$, and $s \in H^0(X,\sE)$ a regular section%
%
, and denote $Z:=Z(s)$ the zero locus of the section $s$. Consider the projectivization $\PP(\sE)$ with natural projection $q: \PP(\sE) \to X$. Since $H^0(X,\sE) \simeq H^0(\PP(\sE), \sO_{\PP(\sE)}(1))$, the section $s$ corresponds to a section $f_s$ of $\sO_{\PP(\sE)}(1)$ on $\PP(\sE)$, hence canonically defines a divisor $\shH_s : = Z(f_s) \subset \PP(\sE)$, which comes with projection $\pi: \shH_s \to X$. The general fiber of this projection is a projective space $\PP^{r-2}$, and the fiber dimension of $\pi$ jumps exactly over $Z$: over $x \in X \,\backslash Z$, $\pi^{-1}(x) \subset \PP(\sE \otimes k(x)) \simeq \PP^{r-1}$ is a hyperplane cut out by $\tilde{s}$, while over $z \in Z$, the fiber $\pi^{-1}(z)$ is the whole fiber $\PP^{r-1}\simeq \PP(\sE \otimes k(z))$ of $\PP(\sE)$. If we denote $i: Z \hookrightarrow X$ the inclusion, then its normal sheaf is $\sN_{i} \simeq \sE|_Z$, and it is clear that 
$\pi^{-1}(Z) = \PP(\sN_i)$. The situation is illustrated in a commutative diagram:
\begin{equation}\label{diagram:Cayley}
	\begin{tikzcd}[row sep= 2.6 em, column sep = 2.6 em]
	\PP(\sN_i) \ar{d}[swap]{p} \ar[hook]{r}{j} & \shH_s \ar{d}{\pi} \ar[hook]{r}{\iota} & \PP(\sE) \ar{ld}[near start]{q} 
	\\
	Z \ar[hook]{r}{i}         & X  
	\end{tikzcd}	
\end{equation}

The geometry in above diagram enables one to reduce a problem of a complete intersection subscheme $Z\subset X$ to a problem of the hypersurface $\shH_s \subset \PP(\sE)$, called the {\em (generalized) universal hyperplane section for $Z=Z(s)$}. This is called {\em Cayley's trick}.

\medskip \noindent\textit{Important observation.} If we define the coherent sheaf $\sG$ to be the cokernel of $s$, i.e 
	\begin{equation*} \label{eqn:G}
	 \sO_X \xrightarrow{~s~} \sE \to \sG = {\rm coker}(s) \to 0,
	\end{equation*}
then $\PP(\sG) = \shH_s \subset \PP(\sE)$, with the inclusion exactly induced by the surjection $\sE \twoheadrightarrow \sG$.

\begin{remark} \label{rmk:N_j} Consider the Euler sequence for $\PP(\sE)$:
	$$0 \to \Omega_q(1) \to q^* \sE \to \sO_{\PP(\sE)}(1) \to 0.$$
The inclusion $\PP(\sN_i) \subset \PP(\sE)$ is cut out by the pullback section $q^* s \in H^0(\PP(\sE), q^* \sE)$, however $\shH_s \subset \PP(\sE)$ is cut out by the image of $q^*$ under the canonical surjection $q^*\sE \twoheadrightarrow\sO_{\PP(\sE)}(1)$. Therefore when restricted to $\shH_s$, the section $q^*s$ lifts canonically to a regular section $\tilde{s} \in H^0(\shH_s, \Omega_q(1)|_{\shH_s})$, and the inclusion $j: \PP(\sN_i) \subset \shH_s$ is cut out by $\tilde{s}$. Hence $\sN_j = \Omega_q(1)|_{\PP(\sN_i)} = \Omega_p (1)$, and the Euler sequence for $\PP(\sN_i)$ is equivalent to:
	$$0 \to \sN_j \to p^* \sN_i \to \sO_{\PP(\sN_i)}(1) \to 0.$$
In particular $\det \sN_j = p^* \det \sN_i \otimes \sO_{\PP(\sN_i)}(-1) =  (\pi^* \det \sE \otimes \sO_{\PP(\sE)}(-1))|_{\PP(\sN_i)}$.
\end{remark}

Cayley's trick can be categorified to obtain relations between derived categories $D(Z)$ and $D(\shH_s)$. This is Orlov's formula:

\begin{theorem}[Orlov, {\cite[Prop. 2.10]{Orlov05}}] \label{thm:HPDI} In the above situation, the functors $\RR j_*\,p^*: D(Z) \to D(\shH_s)$ and $\LL \pi^*(-) \otimes \sO_{\shH_s}(k) : D(X) \to D(\shH_s)$ are fully faithful, where $k =1, \ldots, r-1$, $\sO_{\shH_s}(k) :=  \sO_{\PP(\sE)}(k)|_{\shH_s}$, and there is a semiorthogonal decomposition for $\shH_s$:
	$$D(\shH_{s}) = \langle \RR j_* \, p^* D(Z), \LL \pi^* D(X) \otimes \sO_{\shH_s}(1), \ldots , \LL \pi^* D(X) \otimes  \sO_{\shH_s}(r-1) \rangle.$$
\end{theorem}

If we assume the rank $r$ locally free sheaf $\sE$ over $X$ is globally generated, then Orlov's result fits into the picture of {\em homological projective duality (HPD)} of Kuznetsov \cite{Kuz07HPD}. In the case $\sE = L^* \otimes \sO_X(1)$, where $L$ is a vector space, Thm. \ref{thm:HPDI} corresponds to what Thomas called HPD I, see \cite[Prop. 3.6]{RT15HPD}.

\begin{remark} We don't assume $\shH_s$ or $Z$ to be smooth. In particular, above theorem shows the functor $\RR j_* \, p^*$ induce an equivalence of triangulated categories:
	$$D_{\mathrm{sg}}(Z) \simeq D_{\mathrm{sg}}(\shH_s),$$
see {\cite[Thm. 2.1]{Orlov05}}, where $D_{\mathrm{sg}}(X): = D(X) / {\rm Perf}(X)$ denotes the triangulated category of singularities of a scheme $X$, introduced by Orlov in \cite{Orlov04}.
\end{remark}

\subsection{Blowing up, and relation with Cayley's trick}
Suppose $Z$ is a codimension $r \ge 2$ local complete intersection subscheme of a scheme $X$, then the {\em blowing up of $X$ along $Z$} is $\pi: \Bl_Z X : = \PP(\sI_Z) \to X$, where $\sI_Z$ is the ideal sheaf of $Z$ inside $X$. The {\em exceptional divisor} is $i_E : E := \Bl_Z X \times_X Z \hookrightarrow \Bl_Z X$. Since $\sI_Z|_Z = \sN_{Z/X}^\vee$, therefore $E = \PP(\sN_{Z/X}^\vee)$.

\begin{theorem}[Blowing up formula, Orlov \cite{Orlov05}] \label{thm:blow-up}Suppose $Z$ is a codimension $r \ge 2$ smooth subvariety of a smooth projective variety $X$. Let $\Bl_Z X$ be the blowing up of $X$ along $Z$, and denote $i_{E}: E \hookrightarrow \Bl_Z X$ the inclusion of the exceptional locus. We have a Cartesian diagram
	\begin{equation*}
	\begin{tikzcd}[row sep= 2.6 em, column sep = 2.6 em]
	E \ar{d}[swap]{p} \ar[hook]{r}{i_E} & \Bl_Z X \ar{d}{\pi} \\
	Z \ar[hook]{r}{i_Z}         & X 
	\end{tikzcd}	
	\end{equation*}
Then $\LL \pi^*: D(X) \to D(\Bl_Z X)$ and $ \RR i_{E*} \, \LL p^* (-) \otimes \sO(-kE) : D(Z) \to D(\Bl_Z X)$ are fully faithful, $k \in \ZZ$. Denote the image of the latter to be $D(Z)_k$, then
	\begin{align*} D(\Bl_Z X)	& = \langle \LL \pi^* D(X), ~ D(Z)_0, D(Z)_1, \ldots, D(Z)_{r-2} \rangle; \\
				& = \langle D(Z)_{1-r}, \ldots, D(Z)_{-2}, D(Z)_{-1}, ~  \LL \pi^* D(X) \rangle. 			
	\end{align*}
\end{theorem}

\noindent Note $\sO(k)|_E = \sO_{\PP(\sN_{Z/X}^\vee)}(k)$, and the dualizing sheaf of $\Bl_Z X$ is $\omega = \pi^* \omega_X \otimes \sO_{\Bl_Z X}((r-1)E)$, therefore $D(Z)_k \otimes \omega = D(Z)_{k+1-r}$, so one can obtain one from the other in the above two decompositions from Serre duality.

\subsubsection{Mutations for blowing ups}
The following lemma slightly generalizes Carocci-Turcinovic's result {\cite[Prop. 3.4]{CT15}} and \cite[Lem. 3.6]{JL18Bl}.
\begin{lemma} \label{lem:mut:Bl} In the situation of the blowing up formula Thm. \ref{thm:blow-up}, for any $k \in \ZZ$, we have the following canonical isomorphisms:
	\begin{align*} & \mathbf{L}_{D(Z)_k} (\LL \pi^* (\blank) \otimes \sO(-(k+1)E)) \simeq \LL \pi^*(\blank)  \otimes \sO(-kE), \\
	& \mathbf{R}_{D(Z)_k} (\LL \pi^* (\blank) \otimes \sO(-kE)) \simeq \LL \pi^* (\blank) \otimes \sO(-(k+1)E).
	\end{align*}
\end{lemma}

\begin{proof} Consider the (modified) Kuznetsov's rotation functor
	$$\foR: = \mathbf{L}_{D(Z)_0} \circ \otimes \sO(-E): D(\Bl_Z X) \to D(\Bl_Z X).
	$$
We only need to show that the restriction of $\foR$ to the subcategory $\LL \pi^*D(X)$ is the identity functor. In fact, $\foR|_{ \LL \pi^* D(X)} = \Id$ immediately implies the lemma in the case $k=0$:
	$$\mathbf{L}_{D(Z)_0} \circ \otimes \sO(-E)|_{\LL  \pi^*D(X)} = \Id_{\LL \pi^*D(X)}, \quad \text{and} \quad (\otimes \sO(E)) \circ \mathbf{R}_{D(Z)_0}|_{\LL \pi^*D(X)} = \Id_{\LL \pi^*D(X)};$$
Then the case for general $k$ follows from the case $k=0$ by tensoring with $\sO(-kE)$.

In the rest of the proof, we will write {\em derived} functors as {\em underived}. 
To show $\foR|_{ \LL \pi^* D(X)} = \Id$, following \cite{Kuz15}, we compare $\bL_{{i_{E*}\shA}}$ with the mutation functor $\bL_{\shA}$. Denote $\shA = p^* D(Z)$, and $\shA(E): = \shA \otimes \sO(E)$. Then $D(Z)_0 = i_{E*} \shA$. For a subcategory $\shB  \subset \shT$, denote $i_{\shB} : B \hookrightarrow \shT$ the inclusion, $i_{\shB}^*$, $i_{\shB}^!$ for the left and right adjoints. Notice the inclusion $i_{i_{E*} \shA}: \shA \simeq i_{E*} \shA \hookrightarrow D(\Bl_Z X)$ is obtained from the composition $\shA \xrightarrow{i_{\shA}} D(E) \xrightarrow{i_{E*}} D(\Bl_Z X)$, therefore $i_{i_{E*} \shA} = i_{E*} \, i_{\shA}$, and its right adjoint is $i_{i_{E*} \shA}^! = i_{\shA}^! \, i_{E*}^!$. We have a diagram of triangles:
	\begin{equation} \label{eqn:oct}
	\begin{tikzcd}
	 i_{E*} \,  (i_{\shA}\, i_{\shA}^!)\, i_{E}^! \ar{r}{\sim} \ar{d} & i_{i_{E*}\,\shA}\, i_{i_{E*}\,\shA}^! \ar{r} \ar{d} & 0  \ar{d} \\
	  i_{E*} \,(\id) \, i_{E}^! \ar{r}  \ar{d} & \id \ar{r}  \ar{d} & \otimes \sO_{\Bl_Z X}(E) \ar[equal]{d}   \\
	 i_{E*} \, (\bL_{\shA}) \,i_{E}^!   \ar[dashed]{r}& \bL_{i_{E*}\,\shA} \ar[dashed]{r}& \otimes \sO_{\Bl_Z X}(E) 
	\end{tikzcd}
	\end{equation}
The first two columns are exact triangles again by definition of left mutation, and the second row is one of standard exact triangles for divisoral inclusion $E \hookrightarrow \Bl_Z X$. From octahedral axiom again the last row is an exact triangle, and therefore composed with $\otimes \sO_{\Bl_Z X}(-E)$, we have an exact triangle 
	$$ i_{E*} \, \bL_{\shA} \,i_{E}^! \circ \otimes \sO_{\Bl_Z X}(-E) \to  \bL_{i_{E*}\,\shA}  \circ \otimes \sO_{\Bl_Z X}(-E) \to \Id \xrightarrow{[1]}.$$
Now since $i_E^! = i_E^*( - \otimes \sO_{\Bl_Z X}(E))[-1]$, the first term is $i_{E*} \,  \bL_{\shA} \,i_{E}^*[-1]$, and the second term is by definition $\foR$. But $\bL_{\shA} \,i_{E}^*|_{\pi^*D(X)} = 0$, since the image of $i_{E}^* \,\pi^* = p^* \, i_{Z}^*$ is contained in $\shA \equiv p^*D(Z)$ and thus annihilated by $\bL_{\shA}$. Therefore $\foR|_{\pi^*D(X)} = \Id_{\pi^*D(X)}$, and its inverse $ \foR|_{\pi^*D(X)} ^{-1} =  (\otimes \sO(E)) \circ \bR_{i_{E*}\, \shA}|_{\pi^*D(X)}$ is also the identity. 
\end{proof}

\subsubsection{Mutations for universal hyperplane sections}
\begin{lemma} \label{lem:Rot:Hyp}In the situation of Orlov's universal hyperplane theorem Thm. \ref{thm:HPDI}, we have the Kuznetsov's rotation functor:
$$\foR (\blank): = \mathbf{L}_{D(X)_1} ((\blank) \otimes \sO_\shH(1)): D(\shH) \to D(\shH) 
$$
(where $D(X)_k := \LL \pi^* D(X) \otimes \sO_{\shH_s}(k)$ for $k \in \ZZ$) when restricted to the subcategory $\RR j_* \,p^* D(Z)$ is a translation: $\foR|_{\RR j_* \,p^* D(Z)} = [2]$. Therefore 
we have canonical isomorphisms of functors:
	\begin{align*} & \mathbf{L}_{D(X)_k} (\RR j_* \, p^*  (\blank) \otimes \sO_{\shH}(k) )= \RR j_* \, p^* (\blank)  \otimes \sO_{\shH}(k-1) [2], \\
	& \mathbf{R}_{D(X)_k} (\RR j_* \, p^* (\blank) \otimes \sO_{\shH}(k-1) ) = \RR j_* \, p^* (\blank) \otimes \sO_{\shH}(k) [-2].
	\end{align*}
\end{lemma}
\begin{proof} The proof is totally parallel to the blowing up situation, with the role of $i_{E*}$ played by the pulling back factor $\iota^*$, and the role of left adjoint $i_{E}^!$ played by the left adjoint $\iota_{*}$, where $\iota : \shH \hookrightarrow \PP(\sE)$ is the divisor inclusion in diagram (\ref{diagram:Cayley}). One compares $\bL_{D(X)_1}: D(\shH) \to D(\shH)$ with $\bL_{\pi^* D(X) \otimes \sO_{\PP(\sE)}(1)}: D(\PP(\sE)) \to D(\PP(\sE))$ in a similar diagram as (\ref{eqn:oct}): 
\begin{equation}
	\begin{tikzcd}
	 \iota^* \,  (i_{D(X)(1)}\, i_{D(X)(1)}^!)\, \iota_* \ar{r}{\sim} \ar{d} & i_{D(X)_1}\, i_{D(X)_1}^! \ar{r} \ar{d} & 0  \ar{d} \\
	 \iota^* \,(\id) \,\iota_* \ar{r}  \ar{d} & \id \ar{r}  \ar{d} & \otimes \sO_{\shH}(-\shH)[2] \ar[equal]{d}   \\
	 \iota^*\, ( \mathbf{L}_{\pi^* D(X) \otimes \sO_{\PP(\sE)}(1)}) \, \iota_* \ar[dashed]{r}& \LL_{D(X)_1} \ar[dashed]{r}& \otimes \sO_{\shH}(-\shH)[2]
	\end{tikzcd}
	\end{equation}
Therefore $\bL_{D(X)_1} \circ \otimes \sO_{\shH}(\shH)$, when restricted to $j_* \,p^* D(Z)$, is exactly the translation $[2]$.
\end{proof}

\begin{remark}\label{rmk:sod-H-neg} This in particular allows us to extend the decomposition of hyperplane section formula Thm. \ref{thm:HPDI} to negative degrees, i.e. for any $k = 1,2,\ldots, r-1$, one has:
	$$D(\shH_{s}) = \langle D(X)_{k+1-r}  \ldots, D(X)_{0}, ~~\RR j_* \, p^* D(Z), ~~ D(X)_{1}, \ldots , D(X)_{k-1} \rangle,$$
where $D(X)_k$ denotes the image of $D(X)$ under $\LL \pi^*(\blank)  \otimes \sO_{\shH_s}(k)$ for $k \in \ZZ$.
\end{remark}

\subsection{Degeneracy loci} \label{sec:deg} Since a coherent sheaf $\sG$ is always locally a cokernel of a map between vector bundles $\sF \to \sE$, it will be useful to introduce some notations and results on degeneracy loci. Standard references include \cite{Laz04,Otta,B}.

\begin{definition} \label{def:deg} Let $\sF$ and $\sE$ be two vector bundles of rank $f$ and $e$ respectively on a scheme X, and $\sigma: \sF \to \sE$ a map of $\sO_X$ modules. The {\em degeneracy locus of $\sigma$ of rank $k$} is:
	$$D_k(\sigma) := \{ x \in X~|~ \rank \sigma (x) \le k\}.$$
Then $D_k(\sigma) \subset X$ has a natural closed subscheme structure, and $D_{k-1}(\sigma) \subseteq D_{k}(\sigma)$, see \cite[\S 7,2]{Laz04} for more details. The {\em degeneracy locus of $\sigma$} is $X_\sigma := D_{r-1}(\sigma)$, if $\sigma$ is of generic rank $r$, and denote $\mathring{X}_\sigma: = X_\sigma \,\backslash D_{r-2}(\sigma)$.
\end{definition}

\begin{definition} \label{def:sing} Let $\sG$ be a coherent sheaf of rank $r$ over $X$. For an integer $k \in \ZZ$, the {\em degeneracy locus of $\sG$ of rank $> k$} is defined to be 
	$$X^{> k}(\sG): = \{x \in X \mid \rank \sG(x) > k\}.$$
Then $X^{> k}(\sG) \subset X$ has a natural closed subscheme structure given by the  $k$-th Fitting ideal of $\sG$. Notice that if $k \le r-1$, then $X^{> k}(\sG) = X$. The {\em first degeneracy locus} or the {\em singular locus} of $\sG$ is defined to be $X_{\rm sg}(\sG): = X^{>r}(\sG)$. 
\end{definition}

If we assume $\sE$ and $\sF$ be two vector bundles of rank $e$ and $f$ on a regular scheme $X$, $\sigma \in \Hom(\sF,\sE)$, and denote cokernel $\sC_\sigma : = \operatorname{coker}(\sigma)$, i.e.
	$$\sF \xrightarrow{~\sigma~} \sE \to \sC_\sigma \to 0.$$
Then singular locus of $\sC_\sigma$ coincides with the degeneracy locus of $\sigma$, i.e. $\mathrm{Sing}(\sC_\sigma) = X_{\sigma} \subset X$.

\medskip\noindent\textbf{Observation.} If $\sC_\sigma =  \operatorname{coker}(\sF \xrightarrow{\sigma} \sE)$ is the cokernel of a map $\sigma$ between two vector bundles $\sE$, $\sF$ of rank $e$ and $f$. Then for any $x \in X$, $\sC_\sigma(x)$ is the cokernel of $\sigma(x): \sF(x) \to \sE(x)$.
\begin{enumerate} [leftmargin=*]
	\item If $e \ge f$, $\sC_\sigma$ is of rank no less than $e-f$ over $X$. If $e > f$, then the support of $\sG$ is the whole space: $\mathrm{supp} (\sC_\sigma) = X$, and $\PP(\sC_\sigma)$ is (non-empty) of dimension no less than $\dim X + e - f -1$. The fiber dimension of $\PP(\sC_\sigma) \to X$ will jump over degeneracy loci.
 	\item If $e \le f$, then the fiber of the affine cone scheme $C(\sC_\sigma) = \Spec \Sym^\bullet \sC_\sigma$ over a point $x \in D_{e-r}(\sigma)\,\backslash D_{e-r-1}(\sigma)$ is a affine space $\AA_\kk^r$ of dimension $r$, for $r=1,\ldots, e$. Therefore $\PP(\sC_\sigma)$ is supported on $X_\sigma \equiv D_{e-1}(\sigma)$, and the projection $\PP(\sC_\sigma) \to X_{\sigma}$ is an isomorphism over the open subset $X_\sigma \,\backslash D_{e-2}(\sigma)$. The restriction of $\PP(\sC_\sigma) \to X$ over $D_{e-2}(\sigma) \,\backslash D_{e-3}(\sigma)$ is a $\PP^1$-bundle, over $D_{e-3}(\sigma) \,\backslash D_{e-4}(\sigma)$ a $\PP^2$-bundle, etc. In particular, if $\mathring{X}_\sigma \ne \emptyset$, $\PP(\sC_\sigma) \to X_\sigma$ is often a (Springer type) {\em resolution} of $X_\sigma$ (see Lemma \ref{lem:P(E)} below), with discriminant locus $D_{e-2}(\sigma)$.
		\end{enumerate}

\subsubsection*{Bertini type results} The main results of this paper can be applied in combination with some Bertini type results. Assume for this subsection $\kk = \CC$.

\begin{definition} A vector bundle $\sE$ on a scheme $X$ is called {\em globally generated}, if $V = H^0(X,\sE)$ is finite dimensional, and the natural map $V\otimes\sO_X \to \sE$ is surjective. $\sE$ is called {\em ample} (resp. {\em nef}~) if the line bundle $\sO_{\PP(\sE)}(1)$ is ample (resp. nef) on $\PP(\sE)$.
\end{definition}

For example, a vector bundle $\bigoplus_{i} \sO_{\PP^1}(a_i)$ on $\PP^1$, $a_i \in \ZZ$ is ample if and only if $a_i >0$, for all $i$. Note that $\sE \oplus \sF$ is globally generated (resp. ample) if and only if both $\sE$ and $\sF$ are globally generated (resp. ample). For more about positivity of vector bundles, see Lazarsfeld's book \cite[Chap. 6\&7]{Laz04}.

\begin{lemma} \label{lem:D_k} Let $\sE$ and $\sF$ be two vector bundles of ranks $e$ and $f$ on a regular scheme $X$.
	\begin{enumerate}[leftmargin=*]
	\item If $\sHom(\sF,\sE) = \sF^\vee \otimes \sE$ is {\em globally generated}. Then for a {\em general} morphism $\sigma \in \Hom_X(\sF,\sE)$, $D_k(\sigma)$ is either empty, or have expected codimension $(e-k)(f-k)$, and the singular locus $\mathrm{Sing}(D_k (\sigma))$ satisfies $\mathrm{Sing}(D_k (\sigma)) \subset D_{k-1}(\sigma)$. In particular, if $\dim X <  (e-k+1)(f-k+1)$, then $D_k(\sigma)$ is (either empty or) smooth. 
	\item If $X$ is irreducible, projective, and $\sHom(\sF,\sE) = \sF^\vee \otimes \sE$ is {\em ample}. If $\dim X \ge (e-k)(f-k)$, then $D_k(\sigma)$ is non-empty. If $\dim X > (e-k)(f-k)$, then $D_k(\sigma)$ is connected.
	\end{enumerate} 
In particular, if $X$ is irreducible, projective, $\sHom(\sF,\sE) = \sF^\vee \otimes \sE$ is globally generated and ample, and $(e-k)(f-k) <  \dim X < (e-k+1)(f-k+1)$. Then for a {\em general} morphism $\sigma$, $D_k(\sigma)$ is an irreducible smooth (non-empty) subscheme of $X$. 
\end{lemma}

\begin{proof} See \cite[\S 4.1]{Ba} or \cite[Thm. 2.8]{Otta}, and \cite[Thm. 7.2.1]{Laz04}.
\end{proof}

\begin{lemma} \label{lem:P(E)} Assume $\sHom(\sF,\sE) = \sF^\vee \otimes \sE$ is {\em globally generated}, and $\kk =\CC$. Then for a general section $\sigma \in \Hom_X(\sF,\sE) = H^0(X,\sF^\vee \otimes \sE)$, $\PP(\sC_\sigma) \subset \PP(\sE)$ is a smooth subscheme.
\end{lemma}

\begin{proof} See \cite[Thm. 2]{Ba}, \cite[Lem. 3.1]{BW96}. The key is again to use Cayley's trick \S \ref{sec:semi-local}. Since $\sigma \in \Hom_X(\sF,\sE) = \Hom_{\PP(\sE)}(\pi^* \sF, \sO_{\PP(\sE)} (1))$, where $\pi: \PP(\sE) \to X$ is the projection. Therefore $\PP(\sC_\sigma) \subset \PP(\sE)$ is exactly cut out by the section of the vector bundle $\pi^* \sF^\vee \otimes \sO_{\PP(\sE)} (1)$ which corresponds to $\sigma$ under above identification. Since $\pi^* \sF^\vee \otimes \sO_{\PP(\sE)} (1)$ is globally generated, Bertini's theorem implies the smoothness results.
\end{proof}

\section{Projectivization formula}

Let $\sE$, $\sF$ be two locally free sheaves  of rank $e$ and $f$ on a regular scheme $X$ over $\kk$, and let $\sigma \in \Hom_X(\sF,\sE)$ be a morphism. Then there are canonical identifications:
	$$\Hom_X(\sF,\sE) = \Hom_X(\sE^\vee, \sF^\vee) = H^0(X, \sF^\vee \otimes \sE).$$ 
Therefore $\sigma$ corresponds canonically a map $\sigma^\vee : \sE^\vee \to \sF^\vee$. 

\begin{theorem}\label{thm:duality} In the above situation, consider the following coherent sheaves:
	\begin{align*}
	\sC_\sigma := \operatorname{coker} \big( \sF \xrightarrow{\sigma} \sE  \big), \qquad \sC_{\sigma^\vee}  := \operatorname{coker} \big(\sE^\vee \xrightarrow{\sigma^\vee} \sF^\vee \big).
	\end{align*}
Suppose $\PP(\sC_\sigma)$ and $ \PP(\sC_{\sigma^\vee})$ are of expected dimensions, i.e.
	\begin{equation} \label{eqn:expdim} \dim \PP(\sC_\sigma) = \dim X + e -f -1, \qquad \dim \PP(\sC_{\sigma^\vee}) = \dim X + f - e - 1.\end{equation}
Denote $\pi:  \PP(\sC_\sigma) \to X$ and $\pi': \PP(\sC_{\sigma^\vee}) \to X$ the projections.
	\begin{enumerate}[leftmargin=*]
	\item If $e \ge f$, then the functor $\LL \pi^*(\blank) \otimes \sO_{\PP(\sC_{\sigma})}(k) \colon D(X) \to D(\PP(\sC_{\sigma}))$ is fully faithful for each $k \in \ZZ$. Furthermore, there is a fully faithful Fourier-Mukai functor $\Phi_\shP \colon D(\PP(\sC_{\sigma^\vee})) \to D(\PP(\sC_\sigma))$ and a semiorthogonal decomposition:
	\begin{equation*}
	D(\PP(\sC_\sigma)) = \big \langle \Phi_{\shP} (D(\PP(\sC_{\sigma^\vee}))), D(X)(1), D(X)(2), \ldots,  D(X)(e-f)\big \rangle,
	\end{equation*}
where $D(X)(k)$ denotes the image $\LL \pi^*(D(X))\otimes \sO_{\PP(\sC_{\sigma})}(k)$, $k = 1, \ldots, e-f$, 
	\item If $f \ge e$, then the functor $\LL \pi'^*(\blank) \otimes \sO_{\PP(\sC_{\sigma^\vee})}(k) \colon  D(X) \to D(\PP(\sC_{\sigma^\vee}))$ is fully faithful for each $k \in \ZZ$. Furthermore, there is a fully faithful Fourier-Mukai functor $\Phi_{\shP'} \colon D(\PP(\sC_{\sigma})) \to D(\PP(\sC_{\sigma^\vee}))$ and a semiorthogonal decomposition:
	\begin{equation*}
	D(\PP(\sC_{\sigma^\vee})) = \big \langle  \Phi_{\shP'}(D(\PP(\sC_{\sigma}))), D(X)(1), D(X)(2), \ldots,  D(X)(f-e)\big \rangle.
	\end{equation*}
	\end{enumerate}
If furthermore, $\PP(\sC_\sigma) \times_X \PP(\sC_{\sigma^\vee})$ is of expected dimension 
		\begin{equation} \label{eqn:expdim2} \dim \PP(\sC_\sigma) \times_X \PP(\sC_{\sigma^\vee}) = \dim X -1,
		\end{equation}
then the fully faithful functor $\Phi_\shP$ (resp. $\Phi_{\shP'}$) defined above is given by:
	 $$\Phi_\shP = \RR q_{1*} \, \LL q_{2}^*,
	 \qquad \text{(resp.} 
 	\qquad \Phi_{\shP'} = \RR q_{2*} \, \LL q_{1}^*,)
 	$$
where $q_1$ and $q_2$ denote respectively the projections of $\PP(\sC_\sigma) \times_X \PP(\sC_{\sigma^\vee})$ to $\PP(\sC_\sigma) $ and $\PP(\sC_{\sigma^\vee})$.
\end{theorem}

The theorem states: (since the situation is symmetric, assume $e\ge f$) the derived category  $D(\PP(\sC_{\sigma}))$, as $\PP(\sC_\sigma)$ is a {\em generic $\PP^{e-f-1}$-bundle} over $X$, contains {\em $(e-f)$ copies of $D(X)$}, whose orthogonal component is exactly given by $D(\PP(\sC_{\sigma^\vee}))$. Notice that $\PP(\sC_{\sigma^\vee}) = : \tX$ is  {\em Springer type partial desingularization} of the degeneracy locus $X_\sigma$ (Def. \ref{def:deg}), given by 
	$$\tX = \PP(\sC_{\sigma^\vee})= \{(x, [H_x]) \mid \im \sigma^\vee(x) \subset H_x \} \subset \PP(\sF^\vee),$$
where $x \in X$, and $H_x \subset \sF^\vee(x)$ is a hyperplane.

Interesting special cases of the theorem include:
\begin{itemize}
		\item If $X_\sigma = \emptyset$, then $\PP(\sC_{\sigma^\vee}) = \emptyset$, and $\sC_{\sigma}$ is locally free. The theorem is nothing but Orlov's  {projective bundle} formula Thm. \ref{thm:proj_bundle}. 
		\item If $\sF = \sO_X$, the the expected dimension condition is equivalent to $\sigma: \sO \to \sE$ is a regular section. Then $\PP(\sC_\sigma) = \shH_s$ and $\PP(\sC_{\sigma^\vee}) = Z(s) \subset X$. The theorem is nothing but Orlov's result on generalized {universal hyperplane sections} Thm. \ref{thm:HPDI}. 
		\item If $e =f$, then $X_\sigma^+ = \PP(\sC_\sigma)$ and $\tX= \PP(\sC_{\sigma^\vee})$ are two (in general different) partial desingularizations of the degeneracy locus $X_\sigma$. The theorem states that $D(X_\sigma^+) \simeq D(\tX)$ for the flop $X_\sigma^+ \dashrightarrow \tX$. Examples include Atiyah flops. See \S \ref{sec:springer}.
		\item If $f = e-1$, then the theorem yields a blowup formula for blowups of a regular scheme $X$ along a Cohen-Macaulay subscheme $Z = X_\sigma$ of codimension $2$. See \S \ref{sec:CM2}.
	\end{itemize}		

Therefore above theorem is a {generalization} of Orlov's {projective bundles} formula \cite{Orlov92} and {universal hyperplane section} formula Thm. \ref{thm:HPDI}, as well as the derived equivalences of the flops from {two Springer type partial desingularizations} of determinantal hypersurfaces.

\begin{remark}\label{rmk: expect} Assume $e \ge f$. If each degeneracy locus $D_{f-k}(\sigma)$ satisfies
 	$$\mathrm{codim} (D_{f-k}(\sigma) \subset X) \ge e -f + k, \quad \text{for} \quad k \ge 1,$$
(in particular, $X_{\sigma} = D_{f-1}(\sigma)$ has expected codimension $e-f+1$), then condition (\ref{eqn:expdim}) of the theorem is satisfied. If furthermore the following holds:
 	$$\mathrm{codim} (D_{f-k}(\sigma) \subset X) \ge e -f + 2k - 1, \quad \text{for} \quad k \ge 1.$$
Then $\PP(\sC_\sigma) \times_X \PP(\sC_{\sigma^\vee})$ also has the expected dimension \eqref{eqn:expdim2}. These required codimensions are smaller than the expected codimension $k (e -f + k)$ of $D_{f-k}(\sigma)$ when $k > 1$.
 \end{remark}
  
 \begin{remark}[Generic expected dimension and smoothness]\label{rmk: generic} If $\kk = \CC$, and $\sHom(\sF,\sE)$ is {globally generated}, then the Bertini type results of Lem. \ref{lem:D_k} implies that for a {\em general} $\sigma \in \Hom(\sF,\sE)$, the conditions \eqref{eqn:expdim}, \eqref{eqn:expdim2} of the theorem are satisfied. Moreover, by Lem. \ref{lem:P(E)}, for a general $\sigma \in \Hom(\sF,\sE)$, $\PP(\sC_\sigma)$ and $\PP(\sC_{\sigma^\vee})$ are smooth. Hence, say if $e > f$, then $\PP(\sC_{\sigma^\vee})$ is the {(Springer type) {\em resolution}} of the first degeneracy locus $X_{\rm sg}(\sC_\sigma)$ of $\sC_\sigma$.
 \end{remark}

\begin{proof}[Proof of Theorem \ref{thm:duality}.] Since the situation is symmetric, assume without loss of generality that $e \ge f$. Notice that we have canonical identifications:
\begin{align*}
	& H^0(X, \sF^\vee \otimes \sE) = H^0(\PP(\sE), a_1^*\, \sF^\vee \otimes \sO_{\PP(\sE)}(1)) \\
	& = H^0(\PP(\sF^\vee), a_2^*\, \sE \otimes \sO_{\PP(\sF^\vee)}(1)) = H^0(\PP(\sE) \times_X \PP(\sF^\vee),  \sO_{\PP(\sE)}(1) \boxtimes_X \sO_{\PP(\sF^\vee)}(1)),
\end{align*}
where $a_1 \colon \PP(\sE) \to X$ and $a_2 \colon \PP(\sF^\vee) \to X$ are the natural projections. Therefore $\sigma$ canonically corresponds to sections $s_1$, $s_2$ of vector bundles $ a_1^*\, \sF^\vee \otimes \sO_{\PP(\sE)}(1)$ and $a_2^*\, \sE \otimes \sO_{\PP(\sF^\vee)}(1)$ on $\PP(\sE)$ and respectively $\PP(\sF^\vee)$. Then $\PP(\sC_\sigma) \subset \PP(\sE)$ (resp. $\PP(\sC_{\sigma^\vee}) \subset \PP(\sF^\vee)$) is exactly cut out by the section $s_1$ (resp. $s_2$). Therefore $\PP(\sC_\sigma)$ (resp. $\PP(\sC_{\sigma^\vee})$) is of {\em expected dimension} if and only if $s_1$ (resp. $s_2$) is a {\em regular section}.

Note also $\sigma$ corresponds to a section $s$ of the line bundle $\sO_{\PP(\sE)}(1) \boxtimes_X \sO_{\PP(\sF^\vee)}(1))$ on $\PP(\sE) \times_X \PP(\sF^\vee)$. Denote $\shH \subset \PP(\sE) \times_X \PP(\sF^\vee)$ the zero locus of $s$. Then $\shH$ will play a similar role as {\em universal hyperplane section} in HPD theory. In fact $\shH$ satisfies a fiberwisely universal quadratic relation: over a point $x \in X$, the fiber of $\shH$ is
	$$\shH|_{x} = \{  ([v], [w]) \mid \langle \sigma^\vee(x) \cdot v, w \rangle = \langle v, \sigma(x)\cdot w \rangle = 0 \} \subset  \PP(\sE(x)) \times \PP(\sF^\vee(x)),$$
where $v \in \sE(x)^*$, $w \in \sF(x)$, and the first bracket $\langle \blank, \blank \rangle$ is the pairing between $\sF(x)^*$ and $\sF(x)$, and the second is the one between $\sE(x)^*$ and $\sE(x)$.

Now the key is to observe the following: the space $\shH$ is the common total space of Cayley's trick (\S \ref{sec:semi-local}) for the both the two zero loci of regular sections $\PP(\sC_\sigma) = Z(s_1) \subset \PP(\sE)$ and respectively $\PP(\sC_{\sigma^\vee})  = Z(s_2) \subset \PP(\sF^\vee)$:
\begin{align*}
	\shH & = \PP_{\PP(\sE)}\big(\operatorname{coker}\left[\sO \xrightarrow{s_1} a_1^* \sF^\vee \otimes \sO_{\PP(\sE)}(1)\right] \big) \subset \PP(\sE) \times_X \PP(\sF^\vee), \\
	& = \PP_{\PP(\sF^\vee)}\big(\operatorname{coker}\left[\sO \xrightarrow{s_2} a_2^* \sE \otimes \sO_{\PP(\sF^\vee)}(1)\right] \big) \subset \PP(\sE) \times_X \PP(\sF^\vee).
\end{align*}

The geometry is summarized in the following diagrams, with all the notations of maps as indicated. Notice the three squares (rhombus-shaped) are Cartesian squares, of which the two on each side are exactly ones for Cayley's trick (\ref{diagram:Cayley}) for the above two projectivization.

\begin{equation}\label{diagram}
\begin{tikzcd}[back line/.style={}]		
		&	& \PP(\sC_\sigma) \times_X \PP(\sC_{\sigma^\vee})  \ar[hook']{ld}[swap]{\iota_2}  \ar[hook]{rd}{\iota_1}  \ar[bend right]{ddll}[swap]{q_1} \ar[bend left]{ddrr}{q_2}  \\
		&  \PP(\sC_\sigma) \times_X \PP(\sF^\vee)   \ar{ld}[swap]{p_1} \ar[back line, hook]{rd}[swap]{j_1} &  &  \PP(\sE) \times_X \PP(\sC_{\sigma^\vee}) \ar[crossing over, hook']{ld}{j_2} \ar{rd}{p_2} \\
	\PP(\sC_\sigma) \ar[hook]{rd}{i_1}	&	& \shH  \ar{ld}[swap]{\pi_1} \ar{rd}{\pi_2}	&	& \PP(\sC_{\sigma^\vee}) \ar[hook']{ld}{i_2}\\
		& \PP(\sE) \ar{r}{a_1}	& X	&  \PP(\sF^\vee) \ar{l}[swap]{a_2} 	
\end{tikzcd}
\end{equation}
\medskip

Then Orlov's result Thm. \ref{thm:HPDI} implies there are $X$-linear Fourier-Mukai functors $\Phi_{\shE_1}$ and $\Phi_{\shE_2}$ that embed $D(\PP(\sC_\sigma)$ and $D(\PP(\sC_{\sigma^\vee}))$ into a common category $D(\shH)$, where 
	$$\Phi_{\shE_1} = \RR j_{1*}\, \LL p_1^*:  D(\PP(\sC_\sigma)) \hookrightarrow D(\shH), \quad \Phi_{\shE_2} = \RR j_{2*}\, \LL p_2^*:  D(\PP(\sC_{\sigma^\vee})) \hookrightarrow D(\shH),$$
and $j_1,p_1,j_2,p_2$ are indicated in (\ref{diagram}). The kernels of $\Phi_{\shE_1}$ and $\Phi_{\shE_2}$ are given by 
	$$\shE_1 = \sO_{\PP(\sC_\sigma)\times_X \PP(\sF^\vee)} \in D(\PP(\sC_\sigma)\times \PP(\sF^\vee)), \quad \shE_2 = \sO_{\PP(\sC_{\sigma^\vee}) \times_X \PP(\sE)} \in D(\PP(\sC_{\sigma^\vee}) \times \PP(\sE)).$$
Furthermore, there are $X$-linear semiorthogonal decompositions:
	\begin{align*}
	D(\shH) & = \big \langle \Phi_{\shE_1} D(\PP(\sC_\sigma)), ~ \pi_1^*D(\PP(\sE))(0,1), \ldots, \pi_1^* D(\PP(\sE))(0,f-1)\big \rangle \\
	& = \big \langle \Phi_{\shE_2} D(\PP(\sC_{\sigma^\vee})), ~ \pi_2^*D(\PP(\sF^\vee))(1,0), \ldots, \pi_2^* D(\PP(\sF^\vee))(e-1,0) \big \rangle,
	\end{align*}
where we denote the line bundle $\sO(\alpha,\beta) := \sO_{\PP(\sE)}(\alpha) \boxtimes_X \sO_{\PP(\sF^\vee)}(\beta) |_\shH$, $\alpha, \beta \in \ZZ$, and $\shA(\alpha,\beta)$  denotes the image of a subcategory $\shA$ under the autoequivalences $\otimes \sO(\alpha,\beta)$ of $D(\shH)$.

Since $\shH$ is a {\em $\sO(1,1)$-divisor} inside $\PP(\sE) \times_X \PP(\sF^\vee)$, we are exactly in the situation of a {\em ``chess game''} in \cite{RT15HPD} and \cite{JLX17}, 
see Appendix \ref{sec:app}. More explicitly, denote $i_{\shH}\colon \shH \hookrightarrow \PP(\sE) \times_X \PP(\sF^\vee)$ the inclusion, from the exact sequence $0 \to \sO(-1,-1) \to \sO \to \sO_{\shH} \to 0$ and projection formula, one directly obtains that the cotwist functor of $i_{\shH}*$ satisfies:
	$$ C_{i_{\shH}*} \equiv \cone (1 \to i_{\shH *} i_{\shH}^*)[-1] = \otimes \sO(-1,-1)  \in {\rm Auteq} (D(\PP(\sE) \times_X \PP(\sF^\vee))),$$
this verifies condition (\textbf{A-1}) of  \S \ref{sec:app}. Notice that for $\alpha, \beta \in \ZZ$, one has
	$$\pi_1^* D(\PP(\sE))(0,\beta) = i_{\shH}^* \langle D(X)(k,\beta) \mid k \in \ZZ\rangle, \qquad \pi_2^* D(\PP(\sF^\vee))(\alpha,0) =  i_{\shH}^* \langle D(X)(\alpha, k) \mid k \in \ZZ\rangle,$$
therefore the condition (\textbf{A-2}) is verified by the extended Orlov's hyperplane section theorem (see Rmk. \ref{rmk:sod-H-neg}). Denote $\Phi_{\shE_1}^L: D(\shH) \to D(\PP(\sC_\sigma))$ the left adjoint of the embedding $\Phi_{\shE_1}$, then Thm. \ref{thm:CG} implies the functors	
\begin{align*} & \Phi_{\shE_1}^L \circ \Phi_{\shE_2} \colon D(\PP(\sC_{\sigma^\vee})) \hookrightarrow D(\shH) \to D(\PP(\sC_\sigma)), \\ 
		& \Phi_{\shE_1}^L| \colon  D(X)(i,0) \hookrightarrow D(\shH) \to D(\PP(\sC_\sigma)), \quad i=1,2,\ldots, e-f,
	\end{align*} 
are all fully faithful, and their images give the desired decomposition of $D(\PP(\sC_\sigma))$. For description of Fourier-Mukai kernels, notice $\Phi_{\shE_1}^L$ is given by kernel:
	$$\shE_1^L = \sO_{\PP(\sC_\sigma)\times_X \PP(\sF^\vee)} (0,-f)[f-1] \otimes p_1^* \det \sF^\vee \in D(\PP(\sC_\sigma)\times \PP(\sF^\vee)).$$
So if we right-compose the $\Phi_{\shE_1}^L$ with autoequivalence $\otimes \sO_{\shH} (0,f)[1-f]$ of $D(\shH)$, and left-compose with the autoequivalence $\otimes \det \sF$ of $D(\PP(\sC_\sigma))$, then we have decomposition
	$$D(\PP(\sC_\sigma)) = \big \langle \Phi_{\shP} \big( D(\PP(\sC_{\sigma^\vee})) \big)~, \Phi \big(D(X)(1,0)\big), \ldots, \Phi \big(D(X)(e-f, 0) \big) \big \rangle,$$
where $\Phi_\shP$ and $\Phi|_{D(X)(k,0)}$, $k=1,\ldots, e-f$ are fully faithful, and $\Phi_\shP, \Phi$ are given by:
	\begin{align*}\Phi_{\shP} & = (\otimes \det \sF) \circ \Phi_{\shE_1}^L \circ (\otimes \sO_{\shH} (0,f)[1-f]) \circ \Phi_{\shE_2} \colon & D(\PP(\sC_{\sigma^\vee})) \to D(\PP(\sC_\sigma)),  \\
	\Phi&= (\otimes \det \sF ) \circ \Phi_{\shE_1}^L \circ (\otimes \sO_{\shH} (0,f)[1-f])\colon & D(\shH) \to D(\PP(\sC_\sigma)).
	\end{align*}
By direct computations, $\Phi = \RR p_{1*} \, \LL j_1^*$. Since the maps of diagram (\ref{diagram}) are compatible with their projections to $X$, then one has $\Phi( D(X)(k,0)) =\LL \pi^*(D(X))\otimes \sO_{\PP(\sC_{\sigma})}(k)$, for $k \in \ZZ$.

Now if $\dim \PP(\sC_\sigma) \times_X \PP(\sC_{\sigma^\vee}) = \dim X -1$, then the middle rhombus-shaped square of diagram (\ref{diagram}) is Tor-independent, since it is a square coming from a local complete intersection inside Cohen-Macaulay variety of expected dimensions (see \cite[Lem. 2.32 (iii)]{Kuz07HPD}). Then it follows from base-change formula for Tor-independent squares (see \cite[IV 3.1]{SGA6} or \cite{Kuz07HPD}) that
	$$ \Phi_{\shP} = \RR p_{1*} \, \LL j_1^* \, \RR j_{2*} \, \LL p_2^*  =  \RR p_{1*} (\RR \iota_{2*}  \, \LL \iota_1^*) \LL p_2^*  = \RR q_{1*} \, \LL q_2^*. \qquad \qquad \qquad \qedhere $$
\end{proof}

\begin{theorem}[Projectivization formula] \label{cor:projectivization}
Let $\sG$ be a coherent sheaf of rank $r$ and of homological dimension $\le 1$ on a regular scheme $X$. 
Assume that $\PP(\sG)$ and $\PP(\sExt^1(\sG,\sO_X))$ are irreducible of expected dimensions:
	$$\dim \PP(\sG) = \dim X + (r - 1), \qquad \dim \PP(\sExt^1(\sG,\sO_X)) =  \dim X - (r +1),$$
and $\PP(\sExt^1(\sG,\sO_X)) \times_X Y$ is of expected dimension $\dim X -1$. Then the functors
	$$\RR q_{1*} \, \LL q_{2}^* : D(\PP(\sExt^1(\sG,\sO_X))) \to D(\PP(\sG)) \quad \text{and} \quad \LL \pi^*(\blank) \otimes \sO_{\PP(\sG)}(k): D(X) \to D(\PP(\sG))$$ 
are fully faithful, where $q_1$, $q_2$ are the projections of $\PP(\sG) \times_X \PP(\sExt^1(\sG,\sO_X))$ to the first and second factors, $\pi: \PP(\sG) \to X$ is the projection. Furthermore, the images of these functors induce a semiorthogonal decomposition:
	$$D(\PP(\sG)) = \big \langle \RR q_{1*} \, \LL q_{2}^*( D(\PP(\sExt^1(\sG,\sO_X)))), D(X)(1), \ldots, D(X)(r)\big \rangle,$$
where $D(X)(k)$ denotes the image $\LL \pi^*(D(X)) \otimes \sO_{\PP(\sG)}(k)$, $k=1,\ldots, r$.
\end{theorem}

\begin{proof} Since $X$ is regular (and noetherian, separated over $\kk$), it has the resolution property. Therefore $\sG$ admits a two-term resolution of finite locally free sheaves $0 \to \sF \to \sE \to \sG \to 0$. Hence the theorem reduces to Thm. \ref{thm:duality}.
\end{proof}

By Remark \ref{rmk: expect}, the condition of Thm \ref{cor:projectivization} is satisfied if for all $k \ge 1$,
	$$\mathrm{codim}_X (X^{>r+k-1}(\sG)) \ge r + 2k-1 \quad \text{for} \quad k \ge 1, $$
which is smaller than the expected codimension $k(r+k)$ of $X^{>r+k-1}(\sG) \subset X$ if $k >1$.


\begin{remark} The morphism $p: \PP(\sExt^1(\sG,\sO_X)) \to \mathrm{Sing}(\sG)$ is a Springer type partial desingularization. If we choose a two-term resolution $0 \to \sF \xrightarrow{\sigma} \sE \to \sG$, then
	$$ \PP(\sExt^1(\sG,\sO_X)) = \{(x, [H_x]) \mid \im \sigma^\vee(x) \subset H_x \} \subset \PP(\sF^\vee),$$
where $H_x \subset \sF^\vee(x)$ is a hyperplane. 
The map $p$ has the following property: $p$ is an isomorphism over $X \,\backslash X^{>r+1}(\sG)$, and a $\PP^k$-fibre bundle over $X^{>r+k}(\sG) \,\backslash X^{>r+k+1}(\sG)$ for $k \ge 1$.
\end{remark}

\begin{remark} Although $\PP(\sG)$ and $\PP(\sExt^1(\sG,\sO_X))$ are smooth in the ``generic situation", see Rmk. \ref{rmk: generic}), in general, they could be singular. Combined with Orlov's theory of singularity categories \cite{Orlov04}, Thm. \ref{cor:projectivization} implies that there is a semiorthogonal decomposition
	$${\rm Perf}(\PP(\sG)) = \big \langle \RR q_{1*} \, \LL q_{2}^*( {\rm Perf}(\PP(\sExt^1(\sG,\sO_X)))), ~{\rm Perf}(X)(1), \ldots, {\rm Perf}(X)(r)\big \rangle,$$
as well as an equivalence of categories:
	$$\RR q_{1*} \, \LL q_{2}^*: D_{\mathrm{sg}}(\PP(\sExt^1(\sG,\sO_X))) \simeq D_{\mathrm{sg}}(\PP(\sG)).$$
Here, for a noetherian $\kk$-scheme $Y$, ${\rm Perf}(Y)$ denotes its category of perfect complexes, and $D_{\mathrm{sg}}(Y): = D(Y) / {\rm Perf}(Y)$ denotes its category of singularities, see \cite{Orlov04}. (In fact, as $q_1, q_2$ and $\pi$ are local complete intersection morphisms, the functors $\RR q_{1*} \, \LL q_{2}^*$ and $\LL \pi^*(\blank) \otimes \sO_{\PP(\sG)}(k)$ admit both left and right adjoints with  finite cohomological amplitudes, and preserve coherence and perfect complexes. Hence \cite[Prop. 1.10]{Orlov04} could be applied.) As a consequence, in the situation of Thm. \ref{cor:projectivization}, $\PP(\sExt^1(\sG,\sO_X))$ is smooth iff $\PP(\sG)$ is.
\end{remark}

\begin{remark}\label{rmk:structure} If we interpret Orlov's blowing-up formula Thm. \ref{thm:blow-up} as:
	$$D(\PP(\sI_Z)) = \langle D(\PP(\sExt^{r-1}(\sI_Z,\sO_X))(1), \ldots, D(\PP(\sExt^{r-1}(\sI_Z,\sO_X)))(r-1),  D(X) \rangle.$$
Then this can be viewed as the ``dual" situation of Thm. \ref{cor:projectivization}. This  duality between the projectivization formula Thm. \ref{cor:projectivization} and Orlov's blowing-up formula Thm. \ref{thm:blow-up}  can be formulated precisely in the framework of homological projective duality (HPD), see  \cite{JL18proj2}. 
\end{remark}

\subsection{More examples}

\subsubsection{Flops and Springer type resolutions} \label{sec:springer}
If $V$ is a vector space of dimension $d \ge 2$, $X = \End V$, and $\sF = \sE = V \otimes \sO_X$, and $\sigma \in \Hom(\sF,\sE)$ is the tautological section: $\sigma(A) = A$ for $A \in \End V$. Then $X_\sigma = \{ A \in X=\End V \mid \det A = 0\}$ is the determinantal hypersurface, which has (Gorenstein) singularities along higher degeneracy loci. In this case, 
	$$\PP(\sC_\sigma) = \{ (A, [H]) \mid \im A \subset H \} \subset X \times \PP V = X \times \PP_{\mathrm{sub}} (V^*),$$
where $H \subset V$ is a hyperplane, and 
	$$\PP(\sC_{\sigma^\vee}) = \{ (A, [v]) \mid  v  \in  \ker A \} \subset X \times \PP V^* = X \times \PP_{\mathrm{sub}} (V), $$
are two {\em (Springer) crepant resolutions} of $X_\sigma$ (where $v \in V$ is a vector, which corresponds to a hyperplane in $V^*$). Therefore Thm. \ref{thm:duality} states that the two different Springer crepant resolutions $\PP(\sC_\sigma)$ and $\PP(\sC_{\sigma^\vee})$ are derived equivalent. 

\begin{example}If $d=2$, $X_\sigma$ is a singular threefold with an isolated singularity at $0 \in \End V$, and $\PP(\sC_\sigma)$ and $\PP(\sC_{\sigma^\vee})$ are two small crepant resolutions of singularities of $X_\sigma$, with exceptional locus $(-1,-1)$-curves. $\PP(\sC_\sigma) \dashrightarrow \PP(\sC_{\sigma^\vee})$ is the famous {\em Atiyah flop}.
\end{example}

	In general, if $\sF$ and $\sE$ are two vector bundles on a regular scheme $X$ with $\rank \sF= \rank \sE$, $\mathring{X}_{\sigma} \ne \emptyset$, and the condition \eqref{eqn:expdim} of Thm. \ref{thm:duality} holds, i.e. $\dim \PP(\sC_{\sigma}) = \dim \PP(\sC_{\sigma^\vee}) = \dim X -1$. Then $\PP(\sC_{\sigma})$ and $\PP(\sC_{\sigma^\vee})$ are two different partial desingularizations of the determinantal hypersurface $X_\sigma$ defined by the zero locus of the determinant of $\sigma: \sF \to \sE$, and Thm. \ref{thm:duality} implies that $D(\PP(\sC_\sigma)) \simeq D(\PP(\sC_{\sigma^\vee}))$.

\subsubsection{Cohen-Macaulay subschemes of codimension $2$} \label{sec:CM2} Let $X$ be a regular (noetherian, finite dimensional, separated) $\kk$-scheme, $\sigma \colon \sF \to \sE$ an injective $\sO_X$-module map between finite locally free sheaves. Set $\rank \sF =f$ and $\rank \sE =e$. If $f = e -1$, and the degeneracy locus $Z:=X_{\sigma} \subseteq X$ has the expected codimension $2$. Then $Z \subseteq X$ is a Cohen-Macaulay subscheme, whose ideal sheaf $\sI_Z$ has depth $2$ and homological dimension $1$. In this case, Buchsbaum-Eisenbud's acyclic criterion \cite{BE} implies that the sequence
	$$ 0 \to \sF \xrightarrow{~~~\sigma~~~} \sE \simeq \wedge^{f} \sE^\vee \otimes \det \sE \xrightarrow{~~\wedge^{f}(\sigma^\vee) \otimes 1~~}  \det \sF^\vee \otimes \det \sE \to (\det \sF^\vee \otimes \det \sE)|_Z \to 0$$
is exact. Therefore $\sC_\sigma \simeq \sI_Z \otimes  \det \sF^\vee \otimes \det \sE$, and $\PP(\sC_\sigma) \simeq \PP(\sI_Z)$. 

Conversely, for a regular scheme $X$, every Cohen--Macaulay subscheme $Z \subset X$ of codimension two arises in this way. More precisely, since $\dim X = 2 + \dim Z = 2 + {\rm depth}_X (\sO_Z)$, by the Auslander--Buchsbaum formula, the $\sO_X$-module $\sO_Z$ has homological dimension $2$. Hence the Hilbert-Burch theorem \cite[Thm. 20.15]{Ei} can be applied; In particular, $\sI_Z$ has homological dimension $1$, and there exists an injective map $\sigma \colon \sF \to \sE$ of finite locally free sheaves, such that $\rank \sF = \rank \sE - 1$ and $\sI_Z \simeq \sC_\sigma$.

Let $Z \subseteq X$ be a Cohen--Macaulay subscheme of codimension $2$. If we assume $\widetilde{Z}:=\PP(\sExt^1(\sI_Z, \sO_X))$ has the expected dimension, i.e. $\dim \widetilde{Z} = \dim X - 2$. (This condition is equivalent to $\codim X^{> i}(\sI_Z) \ge i+1$ for all $i \ge 1$.) Then by \cite[Prop. 3.2]{ES}, 
	$$\pi \colon \PP(\sI_Z) = \Bl_Z X \to X$$
is the {\em blowup} of $X$ along $Z$. On the other hand, $\widetilde{Z}=\PP(\sExt^1(\sI_Z, \sO_X)) \to Z$ is a {\em Springer type partial desingularization} of $Z$, which is an isomorphism over $Z \,\backslash\, X^{> 2}(\sI_Z)$. Notice also that $\widetilde{Z} \simeq \PP(\sExt^1(\sI_Z, \omega_X)) = \PP(\omega_Z)$, where $\omega_X$ and $\omega_Z$ are the dualizing sheaves. 

Thm. \ref{thm:duality}, combined with above discussion, implies:

\begin{corollary} Let $X$ be a regular scheme, and let $Z \subset X$ be a Cohen--Macaulay subscheme of codimension $2$. Assume that the following expected dimension conditions hold:
	$$\dim \PP(\sI_Z) = \dim X, \qquad \dim \widetilde{Z} = \dim X - 2.$$
Then there are semiorthogonal decompositions: 
	$$D(\Bl_Z X) = \langle \pi^* D(X) ,~ D(\widetilde{Z}))\rangle = \langle D(\widetilde{Z}), ~\pi^* D(X) \otimes \sO(1)\rangle.$$
Furthermore, if the expected dimension (\ref{eqn:expdim2}) is satisfied for $\sI_Z$, which is equivalent to 
	$$\codim X^{> i}(\sI_Z) \ge 2i  \qquad i \ge 1.$$ 
Then the Fourier-Mukai kernel of the fully faithful embedding $D(\widetilde{Z})) \hookrightarrow D(\Bl_Z X)$ is given by the structure sheaf of the fiber product $\Bl_Z X \times_X \widetilde{Z}$.
\end{corollary}

\begin{proof} We only need to show the equivalence of these two semiorthogonal decompositions of $D(\Bl_Z X)$. The condition of the corollary implies that $\PP(\sI_Z) \subseteq \PP(\sF)$ is a local complete intersection subscheme of expected codimension. Hence $\pi \colon \PP(\sI_Z) \to X$ is a locally compete intersection morphism, with relative dualizing complex $\omega_\pi^\bullet = \sO_{\PP(\sI_Z)}(1)$. The desired equivalence on the level of perfect complexes follows by  Grothendieck duality. On the other hand, since the functors of Thm. \ref{thm:duality} preserves coherence and have finite cohomological amplitudes, the result on the level of bounded derived categories then follows.
\end{proof}

This result generalizes Orlov's blowup formula \cite{Orlov05}, Thm. \ref{thm:blow-up} to the case of blowup along a possibly singular Cohen-Macaulay subscheme $Z \subset X$ of codimension two. The copy of $D(Z)$ in Orlov's formula Thm. \ref{thm:blow-up} (in the case of when $Z$ is smooth) is now replaced by the derived category $D(\widetilde{Z})$ of a partial desingularization $\widetilde{Z}$ of $Z$. This formula is very useful in the study of derived categories of moduli spaces, see \S \ref{sec:Hilb}, \S \ref{sec:Voisin}.

\subsubsection{Symmetric powers of curves and $\Theta$-flops
}\label{sec:SymC}

Let $C$ be a smooth projective curve of genus $g \ge 1$, for $k \in \ZZ$ denote by $C^{(k)}$ the {\em $k$-th symmetric power of $C$} (by convention set $C^{(0)} = {\rm point}$; $C^{(k)} = \emptyset$ for $k <0$), and let $\Pic^k(C)$ be the Picard variety of line bundles of degree $k$ on $C$. For any fixed integer $n \ge 0$, consider the Abel-Jacobi map and its involution:
	$$AJ \colon X_+:=C^{(g-1+n)} \to X:=\Pic^{g-1+n}(C) \quad \text{\and} \quad AJ^\vee \colon X_- :=C^{(g-1-n)} \to X $$
	defined by
 	$$AJ \colon D \mapsto \sO(D) \quad \text{and} \quad AJ^\vee \colon  D \mapsto \sO(K_C - D),$$
where $D \in X_{\pm}$ denotes an effective divisor on $C$, $\sO(D)$ is the associated line bundle, and $K_C$ is the canonical divisor of $C$. The fibre of $AJ$ (resp. $AJ^\vee$) over a point $\sL \in X=\Pic^{g-1+n}(C)$ is the linear system $|\sL| =\PP_{\rm sub}(H^0(C,\sL)) = \PP(H^0(C,\sL)^*)$ (resp. $|\sL^\vee(K_C)| =\PP_{\rm sub}(H^1(C,\sL)^*) = \PP(H^1(C,\sL))$). Consider the fibered diagram:
    \begin{equation}\label{diag:AJ}
    	\begin{tikzcd}[row sep=0.8 em, column sep=2.3 em]
    	  	&\widehat{X}: = X_+ \times_{X} X_-   \ar{ld}[swap]{q_1} \ar{rd}{q_2}& \\
	X_+ = C^{(g-1+n)} \ar{rd}[swap]{AJ}& & X_- = C^{(g-1-n)} \ar{ld}{AJ^\vee} \\
		& X= \Pic^{g-1+n}(C).
	\end{tikzcd}
    \end{equation}
In \cite{Tod2}, Toda realized above diagram as a simple wall-crossing of certain stable pair moduli spaces \cite{PT} on Calabi-Yau threefolds; Thus he deduced that $X_+ \dashrightarrow X_-$ is analytically locally a $d$-critical flip \cite{Tod3}. We show that the diagram (\ref{diag:AJ}) also fits into the framework of projectivization formula Thm. \ref{thm:duality}.

Let $D$ be an effective divisor of large degree on $C$. For every $\sL \in \Pic^{g-1+n}(C)$, the exact sequence
    $0 \to \sL \to \sL(D) \to \sL(D)|_D \to 0$ induces a long exact sequence:
    $$0 \to H^0(C, \sL) \to H^0(C, \sL(D)) \xrightarrow{\mu_{D}} H^0(C, \sL(D)|_D) \to H^1(C, \sL) \to 0.$$
Globalizing (the dual of) above sequence yields the desired picture: let $\sL_{\rm univ}$ be the universal line bundle of degree $g-1+n$ on $C \times X$, and $\pr_C, \pr_X$ be obvious projections, then
	$$\sE := (\RR \pr_{X *} (\pr_{C}^* \sO(D) \otimes \sL_{\rm univ}))^\vee \quad \text{and} \quad \sF := (\RR \pr_{X *} (\pr_{C}^* \sO_D(D) \otimes \sL_{\rm univ}))^\vee$$
are vector bundles on $X$ of ranks $e = \deg(D) + n$ and $f = \deg(D)$,  $\sigma = \mu_D^\vee \colon \sF \to \sE$ is the natural map. Then $X_+ =  \PP(\sC_{\sigma}), X_- = \PP(\sC_{\sigma^\vee})$, and expected dimension conditions (\ref{eqn:expdim}, \ref{eqn:expdim2}) are satisfied (Lem. \ref{lem:app:exp.dim}). Then Thm. \ref{thm:duality} implies:

\begin{corollary}[{Toda \cite[Cor. 5.11]{Tod2}}] \label{cor:SymC} In the above situation (\ref{diag:AJ}), for any $n \ge 0$,
	\begin{align*}
	D(C^{(g-1+n)}) & = \big\langle  \RR q_{1*} \, \LL q_{2}^* (D(C^{(g-1-n)})), ~ \\
	 & \LL (AJ)^* (D(\Pic^{g-1+n}(C))) \otimes \sO_{X_+}(1), \ldots, \LL (AJ)^* (D(\Pic^{g-1+n}(C))) \otimes \sO_{X_+}(n)\big\rangle.
	\end{align*}
\end{corollary}
In the case $n=0$, this yields the derived equivalence $D(X_+) \simeq D(X_-)$ of the {\em $\Theta$-flop} $X_+ =C^{(g-1)} \dashrightarrow X_- =C^{(g-1)}$ from two different resolutions $X_{\pm}$ of the theta-divisor $\Theta: = \{\sL \mid H^0(X,\sL) \ne 0\} \subset \Pic^{g-1}(C) = X$. See \S \ref{sec:Theta} for more results on $\Theta$-flops.


\subsubsection{Nested Hilbert schemes 
}\label{sec:Hilb} The following application is discovered by Belmans--Krug \cite{BK19}; We include here for readers' conveniences, and provide more details about the kernels.

Let $S$ be a smooth surface, denote $\Hilb_n=\Hilb_n(S)$ the Hilbert scheme of $n$-points on $S$, i.e. it parametrizes colength $n$ ideals $I_{n} \subset \sO_S$. Denote the {\em nested} Hilbert schemes by: 
	\begin{align*}
	\Hilb_{n,n+1} & = \{(I_{n+1} \subset I_n)  \mid I_{n+1}/I_{n} \simeq \kk(x),  \text{~for some $x \in S$}\} \subset \Hilb_n \times \Hilb_{n+1},
	\\ \Hilb_{n-1,n,n+1} & = \{I_{n+1} \subset I_n \subset I_{n-1} \mid  I_{n+1}/I_{n} \simeq \kk(x),  I_{n}/I_{n-1} \simeq \kk(x), \text{~for some $x \in S$} \}.
	\end{align*}
Let $X = \Hilb_n(S) \times S$, and $Z_n \subset X$ be the universal subscheme of length $n$. Then $X$ is smooth and $Z_n \subset X$ is Cohen--Macaulay of codimension $2$. Results of \S \ref{sec:CM2} can be applied. 
\begin{equation}\label{diag:Hilb}
    	\begin{tikzcd}[row sep=0.7 em, column sep=3 em]
    	  	&\Hilb_{n-1,n,n+1}(S) \ar{ld}[swap]{q_{-}} \ar{rd}{q_{+}}& \\
	 \Hilb_{n-1,n}(S)   \ar{rd}[swap]{\pi_{-}}& & \Hilb_{n,n+1}(S) \ar{ld}{\pi_{+}} \\
		& X= \Hilb_n(S) \times S.
	\end{tikzcd}
    \end{equation}   

The following properties are summarized from Ellingsrud--Str{\o}mme \cite{ES}, Negu{\c{t}} \cite{Neg} and Maulik--Negu{\c{t}} \cite{MN}:

\begin{lemma} \label{lem:Hilb} 
\begin{enumerate}
	\item \label{lem:Hilb-1} $\Hilb_{n,n+1}(S) = \PP(\sI_{Z_n}) = \Bl_{Z_n} (X)$ is smooth of dimension $2n+2$;
	\item  \label{lem:Hilb-2} $ \Hilb_{n-1,n}(S) = \PP(\sExt^1(\sI_{Z_n},\sO_X)) = \PP(\omega_{Z_n})$ is smooth of dimension $2n$;
	\item  \label{lem:Hilb-3} $\Hilb_{n-1,n,n+1}(S) = \Hilb_{n-1,n}(S) \times_{X} \Hilb_{n,n+1}(S)$ is smooth of dimension $2n+1$.
\end{enumerate}
\end{lemma}
\begin{proof} The equalities $\Hilb_{n,n+1}(S) = \PP(\sI_{Z_n})$ and $ \Hilb_{n-1,n}(S)  = \PP(\sExt^1(\sI_{Z_n},\sO_X)) = \PP(\omega_{Z_n})$ follow from \cite{ES}; \eqref{lem:Hilb-1} and \eqref{lem:Hilb-2} follow from \cite[Prop. 6.3]{MN}; \eqref{lem:Hilb-3} is \cite[Prop. 6.8]{MN}.
\end{proof}

Therefore conditions (\ref{eqn:expdim}, \ref{eqn:expdim2}) of Thm. \ref{thm:duality} are satisfied, and our results imply:

\begin{corollary}[cf. Belmans--Krug \cite{BK19}]
 \begin{enumerate}
 \item $\LL \pi_{+}^*$ and $ \RR q_{+*} \LL q_{-}^*$ are fully faithful, and:
\begin{align*} 
	D(\Hilb_{n,n+1}(S))  &= \big \langle \LL \pi_{+}^* D(\Hilb_n(S) \times S), ~~ \RR q_{+*} \LL q_{-}^* D(\Hilb_{n-1,n}(S))  \big \rangle \\
	& = \big \langle \RR q_{+*} \LL q_{-}^* D(\Hilb_{n-1,n}(S)),  ~~ \LL \pi_{+}^* D(\Hilb_n(S) \times S)  \otimes \sL_n \big \rangle,
\end{align*}
where $\sL_n = \sI_{n}/\sI_{n+1} = \sO_{\PP(\sI_{Z_n})}(1)$ is the tautological line bundle on $\Hilb_{n,n+1}$.
\item For an integer $d \in [0,n]$, consider the diagram of the following form:
	\begin{equation*}
	\begin{tikzcd}[row sep=0.7 em, column sep=0 em]
	  & \Hilb_{n-d,n-d+1,n-d+2}  \ar{ld}{q_-} \ar{rd}[swap]{q_+} &  & \cdots \ar{ld}{q_-} \ar{rd}[swap]{q_+} &  & \Hilb_{n-1,n,n+1} \ar{ld}{q_-} \ar{rd}[swap]{q_+}  & \\
	\Hilb_{n-d,n-d+1} & & \Hilb_{n-d+1,n-d+2} & & \Hilb_{n-1,n} & & \Hilb_{n,n+1}.
	\end{tikzcd}
	\end{equation*}
Then the following functors are fully faithful:
	\begin{align*}
	\underbrace{\RR q_{+*} \LL q_{-}^* \cdots \RR q_{+*} \LL q_{-}^*}_{\text{$d$ times}} \colon D(\Hilb_{n-d,n-d+1}) \hookrightarrow D(\Hilb_{n,n+1}), \\
	\underbrace{\RR q_{+*} \LL q_{-}^* \cdots \RR q_{+*} \LL q_{-}^* }_{\text{$d$ times}}\LL \pi_{+}^* \colon D(S \times \Hilb_{n-d})  \hookrightarrow D(\Hilb_{n,n+1}).
	\end{align*}
\item The functors in $(2)$ for $d=0,\ldots, n$ induce a semiorthogonal decomposition:
	\begin{align*} D(\Hilb_{n,n+1})  \simeq   \langle D(\Hilb_n \times S) , D(\Hilb_{n-1}(S) \times S)),   \ldots,   D(S \times S), D(S)  \rangle.
					\end{align*}
\end{enumerate}
\end{corollary}



\subsubsection{Voisin maps}\label{sec:Voisin} 

Let $Y$ be a smooth cubic fourfold not containing any plane, $F(Y)$ be the Fano variety of lines on $Y$, and $Z(Y)$ be LLSvS eightfold \cite{LLSVS17}.  Voisin \cite{Voi16} constructed a rational map $v \colon F(Y) \times F(Y) \dashrightarrow Z(Y)$ of degree six, Chen \cite{Chen18} showed that the Voisin map $v$ can be resolved by blowing up along the incidence locus 
	$$Z = \{(L_1,L_2) \in F(Y) \times F(Y) \mid L_1 \cap L_2 \ne \emptyset\}.$$ 
More precisely, Chen in \cite{Chen18} shows the following:
	\begin{enumerate}
		\item The incident locus $Z \subset X :=F(Y) \times F(Y)$ is Cohen--Macaulay of codimension $2$, and the blowing up variety $\PP(\sI_Z) = \Bl_Z ( F(Y) \times F(Y))$ is a natural relative $Quot$ scheme over $Z(Y)$ if $Y$ is very general.
		\item The degeneracy loci of $\sI_Z$ over $X$ are given by ($X = X^{> 0}(\sI_Z)$, and)
			$$X^{> 1}(\sI_Z)=Z,  \quad X^{> 2}(\sI_Z)=\Delta_2, \quad X^{> i}(\sI_Z) = \emptyset ~\text{for}~ i \ge 3.$$ 
		Here $\Delta_2  = \{L \in \Delta \simeq F(Y) \mid \sN_{L/Y} \simeq \sO(1)^{\oplus 2} \oplus \sO(-1)\}$ is the type II locus, which is an algebraic surface (hence has the expected dimension).
		\item $\sExt^1(\sI_Z,\sO_X) = \omega_Z$, where $\omega_Z$ is the dualizing sheaf of $Z$.
	\end{enumerate}

Therefore the conditions of Cor. \ref{cor:projectivization} is satisfied, and projectivization formula implies:
\begin{corollary}
There is a semiorthogonal decomposition:
	\begin{align*}
	D(\Bl_Z ( F(Y) \times F(Y))) & = \langle D(F(Y) \times F(Y)), D(\widetilde{Z})\rangle  \\
	&=   \langle D(\widetilde{Z}), D(F(Y) \times F(Y))\otimes \sO(1) \rangle,
	\end{align*}
where $\widetilde{Z} = \PP(\omega_Z)$ is a Springer type partial desingularization of the incidence locus $Z$, which is an isomorphism over $Z \backslash \Delta_2$, and a $\PP^1$-bundle over the type II locus $\Delta_2$.
\end{corollary}

In particular this implies $\widetilde{Z}$ is smooth if and only if $\Bl_Z ( F(Y) \times F(Y))$ is. See \cite{J19} for corresponding relations on Chow groups and rational Hodge structures.

\section{Autoequivalences} \label{sec:auto}
In the local situation of projectivization formula (see Thm \ref{thm:duality}), we assume that $e=f=:r$ and $\pX  := \PP(\sC_\sigma)$ and $\tX : = \PP(\sC_{\sigma^\vee})$ are smooth, and the expected dimension condition
	$$\dim \pX = \dim \tX = \dim \pX \times_X \tX  = \dim X -1$$
is satisfied. Then $\pX \dashrightarrow \tX$ is a flop obtained by two crepant resolutions of the degeneracy locus $X_\sigma$ of Springer type (\S \ref{sec:springer}). The projectivization formula (Theorem \ref{thm:duality}) implies that there is an equivalence of categories induced by the functor $\RR q_{2*} \, \LL q_1^{*}: D(\pX) \simeq D(\tX)$, where $q_1$, $q_2$ are the birational maps in the diagram:
	$$ 
	\begin{tikzcd}
		& X_\sigma^+ \times_X \tX  \ar{ld}[swap]{q_1} \ar{rd}{q_2}& \\
		\pX \ar[dashed]{rr}& & \tX
	\end{tikzcd}
	$$
In this section, we show that there is a sequence of equivalences $\{\Phi_k\}_{k \in \ZZ}$ given by:
	\begin{equation}\label{eqn:Phi}\Phi_k: D(X_\sigma^+) \to D(\tX), \quad \shE^\bullet \mapsto \RR q_{2*} \, (\LL q_1^{*}(\shE^\bullet) \otimes \sO(k,k)), \qquad k \in \ZZ.\end{equation}
In particular, $\Phi_0 = \RR q_{2*} \, \LL q_1^{*}$. Moreover, the members of $\{\Phi_k\}_{k \in \ZZ}$ are connected with each other by spherical twists. Here $\sO(a,b)$ denotes the restriction of the line bundle $\sO_{\PP(\sE)}(a) \boxtimes_X \sO_{\PP(\sF^\vee)}(b)$ to $X_\sigma^+ \times_X \tX$, $a,b \in \ZZ$.
	
\begin{theorem}[``Flop--flop=twist"']\label{thm:flop} In the above situation, (i) for each $k \in \ZZ$, the functors $\Phi_k$ defined in (\ref{eqn:Phi}) is an equivalence of categories $\Phi_k: D(X_\sigma^+) \xrightarrow{\sim} D(\tX)$. (ii) The equivalences $\{\Phi_k\}_{k \in \ZZ}$ are connected to each other by:
	$$\Phi_{k-1}^{-1} \circ \Phi_{k} = T^{-1}_{S_{-k}}[2] ,\qquad \Phi_{k}^{-1} \circ \Phi_{k-1} = T_{S_{-k}}[-2],$$
where 
for each $k \in \ZZ$, $S_{k}: D(X) \to D(\pX)$ is a spherical functor defined by:
	$$S_k(\blank) = \LL \pi^{+*} (\blank) \otimes \sO_{\PP(\sE)}(k),\qquad k \in \ZZ,$$
(where $\pi^{+}: \pX \to X$ is the projection), and $T_{S_k} \in {\rm Autoeq} (D(\pX))$ is the twist functor around the spherical functor $S_k$.  (iii) In particular, the ``flop--flop'" functor equals to 
	$$ \RR q_{1*} \, \LL q_2^{*} \,\, \RR q_{2*} \, \LL q_1^{*}  = T_{S_{0}}^{-1} \circ \otimes \sL^\vee [2],$$
where $\sL = \det \sF^\vee \otimes \det \sE$. (Note that tensoring with $\sL^\vee$ commutes with $T_{S_k}^{-1}$.)
\end{theorem}

This result compares nicely with ``flop--flop = twist" result for the Bondal and Orlov's flopping equivalence for Aityah flops, and the more general  ``flop--flop = twist" results for standard flops of Addington, Donovan and Meachan \cite[Theorem A]{ADM}. Furthermore, the theorem implies that the diagram of categories
  \begin{equation*}
      \xymatrix{
D(X_\sigma^+) \ar@<-.7ex>[r]_{\LL q^*_1}& D(X_0)
\ar@<-.7ex>[l]_{\RR q_{1*}}
\ar@<.7ex>[r]^{\RR q_{2*}}& D(\tX)
\ar@<.7ex>[l]^{\LL q^*_2}   
}
\end{equation*} 
represents a perverse schober on $\CC$, in the sense of \cite{KS,BKS}, where $X_0 : = X_\sigma^+ \times_X \tX$.

\begin{proof}
Assume we are in the situation of Thm. \ref{thm:duality}, and all maps are indicated in the diagram (\ref{diagram}). The following simple computations from Rmk. \ref{rmk:N_j} will be useful later:
	\begin{align*}
	&\omega_{p_1} = \det \sF^\vee \otimes \sO(0,-r),  & \omega_{p_2} = \det \sE  \otimes \sO(-r,0),  \\
	& \omega_{j_1} = \det \sF^\vee \otimes \sO(r-1,-1), & \omega_{j_2} = \det \sE \otimes \sO(-1,r-1).
	\end{align*}

For simplicity of notations, in the rest of the proof, all functors are assumed to be {\em derived} but written as underived. Denote $I_1 := j_{1*} \, p_1^* : D(\pX) \hookrightarrow D(\shH)$, $I_2 = j_{2*} \, p_2^*: D(\tX) \hookrightarrow D(\shH)$ the inclusions, and denote $\Psi_k :  D(\pX) \to D(\tX)$ the functors for $k \in \ZZ$:
	$$\Psi_k(\blank) : = I_2^*(I_1(\blank)  \otimes \sO_{\shH}(k,k)) = p_{2!} \, j_2^* ( \sO_{\shH}(k,k) \otimes  j_{1*} \, p_1^*(\blank)).$$
From $p_{2!}(\blank) = p_{2*} (\blank \otimes \sO(-r,0) \otimes p_2^* \det \sE [r-1])$, we obtain the relations:
 	\begin{align*}
		\Psi_k  =  \sO_{\PP(\sF^\vee)}(r) \circ \det \sE  \,[r-1]  \circ \Phi_{k-r} 
		 = \Phi_{k} \circ  \sO_{\PP(\sE)}(-r)  \circ \det \sE[r-1]
	\end{align*}
Similarly, if we introduce the flopping functors in the other direction:
	$$\Phi'_k: D(\tX) \to D(\pX), \quad \Phi'_k(\blank) =  \RR q_{1*} \, (\LL q_2^{*}(\blank) \otimes \sO(k,k)), \qquad k \in \ZZ,$$
then the right adjoint $\Psi_k^R = I_1^!\circ \sO(-k,-k) \circ I_2$ (which will be the inverse) of $\Psi_k$ is
	\begin{align*}
	\Psi_k^{R} 
		 = \sO_{\PP(\sE)}(r) \circ \det \sF^\vee [1-r] \circ \Phi'_{-k-1} 
		 = \Phi'_{r-1-k} \circ \sO_{\PP(\sF^\vee)}(-r) \circ \det \sF^\vee [1-r] 
	\end{align*}
Therefore, to prove the theorem, we only need to show that for any $k \in \ZZ$, $\Psi_k$ is an  equivalence, and the following holds:
	\begin{equation}\label{eqn:Psi:twist}
	\Psi_{k-1}^{-1} \circ \Psi_{k} = T'_{J_{k}} [2], \qquad \Psi_{k}^{-1} \circ \Psi_{k-1} = T_{J_{k}} [-2],\end{equation}
	where $J_{k} = S_{r-k}$; Then the ``flop--flop" functor will be equal to
	\begin{align*} & \RR q_{1*} \, \LL q_2^{*} \,\, \RR q_{2*} \, \LL q_1^{*}\,  = \Phi_0' \circ \Phi_0 = \Psi_{r-1}^{-1} \circ \Psi_r \circ \sL^\vee 
	 = T_{S_0}^{-1} \circ \sL^\vee [2].
	\end{align*}

We prove the desired result by performing mutations on the ``chessboard" diagram \ref{fig_auto}. In the situation of proof of Thm. \ref{thm:duality}, denote $\sD_1 = D(\pX)$, $\sD_2 = D(\tX)$, and $I_i : \sD_i \hookrightarrow D(\shH)$ the inclusions for $i=1,2$, $I_i^*$, $I_i^!$ the left and respectively right adjoints of $I_i$. (Notice that the inclusion functors are given explicitly by $I_i = j_{i*} \, p_i^*: \sD_i \hookrightarrow D(\shH)$ for $i=1, 2$). Denote  $\sigma  := \otimes \sO_{\shH}(1,1) \colon D(\shH) \to D(\shH)$ the autoequivalence of $D(\shH)$ of tensoring the line bundle $\sO(1,1)$ on $\shH$, and denote the subcategories:
	$$\shE(\alpha,\beta): = D(X)\otimes (\sO_{\PP(\sE)}(\alpha) \boxtimes \sO_{\PP(\sF^\vee)}(\beta))|_\shH \subset D(\shH),$$
and 	$\shE(*, \beta) = \langle \shE(k, \beta) \mid k \in \ZZ \rangle \subset D(\shH)$, 
	$\shE(\alpha, *) = \langle \shE(\alpha, k) \mid k \in \ZZ \rangle \subset D(\shH),$
where $\alpha, \beta \in \ZZ$. Notice $\sigma (\shE(\alpha,\beta)) = \shE(\alpha+1,\beta+1)$, $\sigma (\shE(*,\beta)) = \sigma (\shE(*,\beta+1))$, etc. And the decomposition in the proof of Thm. \ref{thm:duality} takes a standard form of a ``chess game'':
	\begin{align*}
	D(\shH)  = \big \langle I_1(\sD_1), ~ \shE(*,1), \ldots,\shE(*, r-1) \big \rangle 
	 = \big \langle I_2(\sD_2), ~\shE(1,*), \dots, \shE(r-1,*)\big \rangle.
	\end{align*}
Therefore the result of Thm \ref{thm:duality} is $I_2^* \circ I_1: \sD_1 \simeq \sD_2$ is an equivalence. But notice since 
	\begin{align*}
	I_2^* \circ \sigma^k  = \LL_{\langle \shE(1,*), \dots, \shE(r-1,*) \rangle} \circ \sigma^k =
	\sigma^k \circ \LL_{\langle \shE(1-k,*), \dots, \shE(r-1-k,*) \rangle} = \sigma^k \circ I_2'^*,
	\end{align*}
where $\sD_2':=\sigma^{-k} \sD_2$, and $I_2'^*$ is the left adjoint of the inclusion $I_2'\colon \sD_2 \hookrightarrow D(\shH)$. Apply the ``chess game" Thm. \ref{thm:CG} to $\sD_1$ and $\sD_2'$, one obtains $I_2'^* \circ I_1 \colon \sD_1 \simeq \sD_2'$. Therefore the functor
	$$\Psi_k =  I_2^* \circ \sigma^k \circ I_1 = \sigma^k \circ (I_2'^* \circ I_1) \colon  \sD_1 \xrightarrow{\sim} \sD_2'  \xrightarrow{\sim} \sD_2$$ is an equivalence for all $k \in \ZZ$. Note that $\Psi_k$ is just the ``parallel transport" of the functor $\Psi_0 = I_2^* \circ I_1$ on the ``chessboard" (Figure \ref{fig_auto}) downwards by $k$ steps.

\smallskip \noindent \textbf{Claim.} (\ref{eqn:Psi:twist}) holds for $\Psi_k$, i.e. $\Psi_k = \Psi_{k-1} \circ T_{J_k}'[2]$, for all $k \in \ZZ$.

\smallskip\noindent {\em Proof of claim.} Since the case for general $k$ just amounts to parallel transport of the ``chessboard" downwards by $k-1$ steps from the case $k=1$ (i.e. replace all $\shE(\alpha,\beta)$ by $\shE(\alpha+1-k, \beta)$), we only need to show the case for $k=1$, that is to show
	$$\Psi_1 = \Psi_0 \circ T_{\shE(0,1)}' [2],$$
where $T_{\shE(0,1)}'$ denotes the dual twist around the composition $J_1  \colon \shE(0,1) \hookrightarrow D(\shH) \xrightarrow{I_1^!} \sD_1$.

First, we can compare $\Psi_1 = I_2^* \circ \sigma \circ I_1$ with right mutation on the chessboard using Lem. \ref{lem:Rot:Hyp}, from which we know $\foR|_{\sD_1} = \bL_{\shE(*,1)} \circ \sigma = [2]$. Therefore $\sigma|_{\sD_1} = \bR_{\shE(*,1)} [2] |_{\sD_1}$, and 
	$$ \Psi_1 = I_2^* \circ \sigma \circ I_1 = I_2^* \circ \bR_{\shE(*,1)} \circ I_1 \,[2] = I_2^* \circ \bR_{\shE(0,1)} \circ I_1 \,[2],$$
where the last equality follows from $I_2^*$ is left mutation passing through $\langle \shE(1,*), \dots, \shE(r-1,*) \rangle$, therefore kills all $\shE(\alpha,1)$ for all $\alpha \ge 1$.

Next, it remains to compare the right mutation functor $\bR_{\shE(*,1)}$ with dual twist functor $T_{\shE(*,1)}'$. The strategy is similar to the proof of Lem. \ref{lem:mut:Bl} and Lem. \ref{lem:Rot:Hyp}. Denote $E: \shE(0,1) \hookrightarrow D(\shH)$ the inclusion, and $E^*$ its left adjoint as usual. Then we have a diagram
	\begin{equation}
	\begin{tikzcd}
	 I_1 \, T_{\shE(0,1)}' \ar[dashed]{r}  \ar{d}& \bR_{\shE(0,1)} \, I_1 \ar{d} &  \\
	 I_1 \, \Id \ar[equal]{r}  \ar{d} &  \id \,I_1\ar{r}  \ar{d} & 0 \ar{d}   \\
	 I_1 I_1^! \, E\,E^* \, I_1 \ar{r}& E\,E^*\,I_1 \ar{r}& \cone(I_1I_1^! \to \Id) (E E^* I_1)
	\end{tikzcd}
	\end{equation}
from the very definition of right mutations and dual twist. Notice $i_1^!$ is the right mutation functor passing through the category $\sD_1^\perp = \langle \shE(*, 2-l), \ldots, \shE(*,0) \rangle$, and from standard mutation computations (see Lem. \ref{lem:app:cone}), the image of $\cone(I_1I_1^! \to \Id)$ on $\shE(0,1)$ belongs to the staircase region 
	$$\big \langle \shE(\alpha,\beta)  \big \rangle_{1 \le \alpha \le r-1, ~\alpha+1-r \le \beta \le 0},$$ 
in particular contained in ${}^\perp \sD_2$, therefore killed by $I_2^*$. The situation is described in figure \ref{fig_auto}.  

\begin{figure}[h]
\begin{center}
\includegraphics[height= 2.6in]{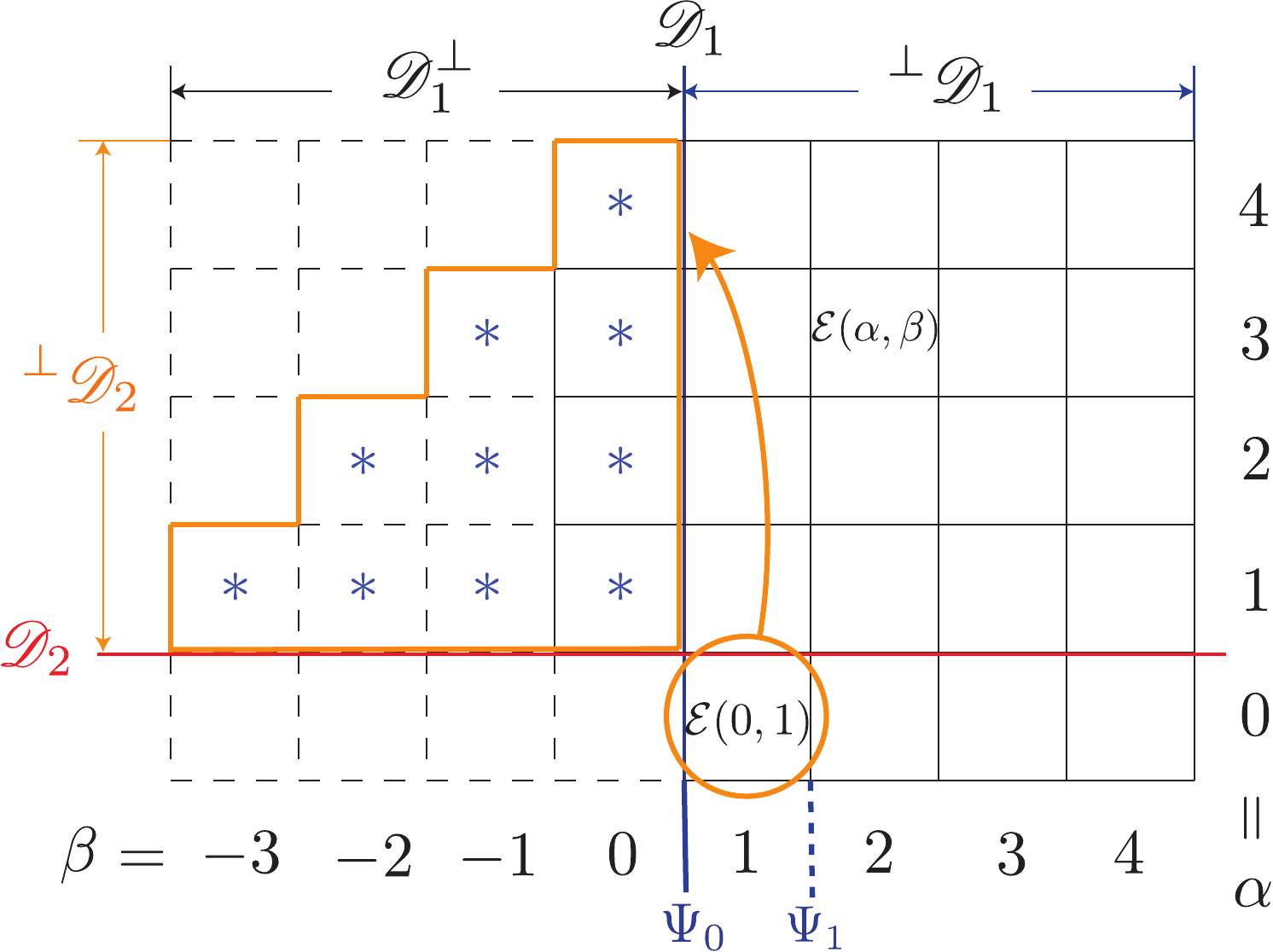}
\caption{The chessboard with boxes indicating $\shE(\alpha,\beta)$ for $\sD_1 = D(\pX)$, $\sD_2 = D(\tX)$, and $I_i : \sD_i \hookrightarrow D(\shH)$ the inclusions, $i=1,2$.  The staircase region marked with `$*$' indicate the image of $\shE(0,1)$ under $\cone(I_1 I_1^! \to \Id)$.} \label{fig_auto}
\end{center}
\end{figure}

Hence $ \Psi_0 \circ T_{\shE(0,1)}' [2] = I_2^* \circ I_1 \circ T_{\shE(0,1)}' [2] \simeq I_2^* \circ \bR_{\shE(0,1)} \circ I_1 \,[2] = \Psi_1. $ \hfill QED of {\bf Claim}.

\medskip

Now the claim is proved, and it remains to compute $J_k$ expliclity. Back to the geometric situation of Thm. \ref{thm:duality}, where the functors $J_k$ are given by the composition:
	$$J_k  \colon D(X) \xrightarrow{E_k} D(\shH) \xrightarrow{I_1^!} D(\pX),$$
and $E_k: D(X) \hookrightarrow D(\shH)$ are the twisted fully faithful embeddings: 		$E_k = a_{\shH}^* (\blank) \otimes \sO_{\shH}(1-k,1),$
where $a_{\shH}: \shH \to X$ is the natural projection. Since $I_1^! = p_{1*} \, j_1^! = p_{1*} \circ j_1^* \circ \det \sF^\vee \otimes \sO(r-1,-1)[1-r]$. Therefore $J_k(\blank) = (\pi^{+*}(\blank) \otimes \sO(r-k,0))  \otimes \det \sF^\vee \, [1-r]$. Notice $\det \sF^\vee \, [1-r]$ comes from autoequivalences from the source, therefore does not effect the twist functor, therefore we can drop them in the expression of $J_k$. Denote $S_{k} = J_{-k+r}$, i.e. $S_{k} = (\pi^{+*}(\blank) \otimes \sO(k,0))$, and (\ref{eqn:Psi:twist}) translates to the desired results for $\Phi_k$ and $S_k$.
\end{proof}

\medskip
Notice although we have proved $T_{S_k}$'s are autoequivalences, we have not shown yet $S_k$'s are indeed {\em spherical}. This follows from the following general observation:
\begin{lemma}\label{lem:spherical_stratified} Let $\pi: Z \to X$ be a map between smooth varieties which can be factorized into a closed immersion $i: Z \hookrightarrow \PP(\sE)$ followed by a smooth morphism $a: \PP(\sE) \to X$, 
	$$
	\begin{tikzcd}
		Z \ar[hook]{r}{i} \ar{rd}[swap]{\pi} & \PP(\sE) \ar{d}{a}\\
			& X,
	\end{tikzcd}
	$$
where $\sE,\sF$ are vector bundles both of rank $r \ge 2$ on $X$, and $i$ is cut out by a regular section $s$ of $a^*\sF^\vee \otimes \sO_{\PP(\sE)}(1)$. Then the image $\overline{Z}: = \pi(Z) \subset X$ is a determinantal hypersurface cut out by a section $\det \sigma$ of the line bundle $\sL$, where $\sL = \det \sF^\vee \otimes \,\det \sE$, $\sigma \in \Hom_X(\sF, \sE) = \Hom_{\PP(\sE)}(a^*\sF, \sO_{\PP(\sE)}(1))$ corresponds to $s$ (i.e. $\overline{Z}  = D_{r-1}(\sigma) = Z(\det \sigma) \subset X$, if we use the notation of \S \ref{sec:deg}). Then for each $k \in \ZZ$, 
\begin{enumerate}
	\item The functor 
	$S_k(\blank) = \LL \pi^* (\blank) \otimes \sO_{\PP(\sE)}(k): D(X) \to D(Z)$
is {\em spherical}.
	\item The cotwist functor
	$C_{S_k} = \cone (\id_X \to \RR\pi_* \LL \pi^*)[-1] \in {\rm Autoeq}(D(X))$ 
is given by $$C_{S_k}(\blank) = (\blank) \otimes \sL^\vee.$$
	\item The twist functor 
	$T_{S_k} = \cone (S_k S_k^R \to \id_Z) \in {\rm Autoeq}(D(Z))$
 is {\em $X$-linear}:
	$$T_{S_k} (\blank \otimes^\LL \LL \pi^* \shF^\bullet) = T_{S_k} (\blank) \otimes^\LL \LL \pi^* \shF^\bullet, \qquad \forall ~\shF^\bullet \in D(X).$$
(Here $S_k^R$ is the right adjoint of $S_k$.) Therefore $T_{S_k}$ only depends on the formal neighborhood of $\overline{Z}$ inside $X$. 
	\item Over the open subscheme $Z \,\backslash \pi^{-1} (D_{r-2}(\sigma)) \subset Z$, 
	$T_{S_k}$ differs from the identity functor by tensoring a line bundle and a degree shift: 
		$$T_{S_k}(\shE^\bullet)= \shE^\bullet \otimes \pi^* \sL^\vee [2], \quad \text{for} \quad \shE^\bullet \in D(Z \,\backslash \pi^{-1}( D_{r-2}(\sigma))).$$	
	\end{enumerate}
\end{lemma}

\begin{proof} For simplicity of notations, the functors in this proof are assumed to be {\em derived}.
\subsubsection*{Proof of (1)$\sim$(3)}
By the projection formula, $\pi_*\, \pi^* (\blank) = (\blank) \otimes \pi_* \sO_Z$. To compute $ \pi_* \sO_Z = a_* \, i_* \sO_Z$, we resolve $i_* \sO_Z$ by the Koszul complex: 
	$$i_* \sO_Z \simeq \shK^\bullet = \left[ \wedge^r \sF(-r) \to \wedge^{r-1} \sF(-r+1) \to \cdots \to \wedge^2 \sF(-2) \to \sF(-1) \to \sO_{\PP(\sE)} \right],$$
where $\sF$ stands for $a^*\sF$, and $(k)$ stands for twisted by $\sO_{\PP(\sE)}(k)$ for $k \in \ZZ$. 
By Grothendieck spectral sequence $R^q a_* (\shK^p) \implies R^{p+q} a_* (\shK^\bullet)$, see e.g. \cite[Rmk. 2.67]{Huy}, we obtain $\pi_*  \sO_Z = [\sL^\vee \xrightarrow{\det \sigma} \sO_X] \simeq \sO_{\overline{Z}}$. Hence there is a triangle
	$ \id_X \to \pi_* \, \pi^* \to \otimes \sL^\vee [1] \xrightarrow{[1]}.$
%
%
Notice that $S_k^R S_k = \pi_*\,\pi^*$, where $S_k^R (\blank) = \pi_* ((\blank) \otimes \sO(-k))$ is the right adjoint of $S_k$. 
Therefore 
	$$C_{S_k} = C_{\pi^*} = \otimes \sL^\vee,$$
which is an autoequivalence of $D(X)$. The left adjoint $S_k^L$ of $S_k$ is
	\begin{align*} 
	S_k^L &= \pi_! \circ \sO(-k) = \pi_* ((\blank) \otimes \sO(-k) \otimes \omega_\pi [\dim \pi]) \\
	& = \pi_*((\blank) \otimes \sO(-k) \otimes \pi^*( \det \sF^\vee \otimes \det \sE) [-1] )= \pi_*( (\blank) \otimes \sO(-k)) \otimes \sL [-1].
	\end{align*}
Therefore $S_k^R = C_{S_k} S_k^L [1]$ holds. By Def. \ref{def:spherical}, $S_k$'s are spherical. Since $S_k S_k^R$ and $\id_Z$ are both $X$-linear, $T_{S_k}$ is also  $X$-linear. Hence we have proved (1)$\sim$(3).

\subsubsection*{Proof of (4)} Notice that the restriction $\pi_0=\pi|_{Z^0}: Z^0  \to X$, where $Z^0 = Z \,\backslash \pi^{-1} (D_{r-2}(\sigma))$, factors through an isomorphism 
$Z^0 \simeq \overline{Z}^{\rm sm}$, where $ \overline{Z}^{\rm sm} :=  \overline{Z} \,\backslash D_{r-2}(\sigma)$, followed by a divisoral embedding $j: \overline{Z}^{\rm sm}  \hookrightarrow X$ defined by a section of the line bundle $\sL$. Therefore $\cone(\pi_0^*\,\pi_{0*} \to 1) (\shE^\bullet \otimes \sO(-k)) = \shE^\bullet \otimes \sO(-k) \otimes \sL^\vee[2]$. Left-composing with $\sO(k)$, we are done. \end{proof}

\begin{example}[Atiyah flop] \label{ex:Atiyah} In the case of Atiyah flop $\pX \dasharrow \tX$ of \S \ref{sec:springer}  ($n=\rank V= 2$), where $\pX \subset X \times \PP^1$ and $\tX \subset X \times (\PP^1)^*$ are two small crepant resolutions of $X_\sigma \subset X = \CC^4$, with exceptional loci $\PP^1$ and respectively $(\PP^1)^*$ over $0 \in X$. The derived equivalence $D(\pX) \simeq \D(\tX)$ is established by Bondal and Orlov \cite{BO}. Furthermore, it is known that ``flop--flop" functor is a spherical twist: $ \RR q_{1*} \, \LL q_2^{*} \,\, \RR q_{2*} \, \LL q_1^{*}\,  = T^{-1}_{\sO_{\PP^1}(-1)} \in {\rm Autoeq}(D(\pX))$, see for example \cite{ADM}. Compared with Thm. \ref{thm:flop}, we have the following equality:
	$$T_{S} = T_{\sO_{\PP^1}(-1)} [2] \in {\rm Autoeq}(D(\pX)),$$
where $S =\LL \pi^{+*}$ is spherical by above lemma ($\pi^{+}: \pX \to X$ is the natural projection). We will provide a direct proof of this fact in Appendix \S \ref{sec:app:Atiyah}, Lemma \ref{lem:app:Atiyah}.
\end{example}

In general if $n >2$, $D_{r-1}(\sigma) \subset \overline{Z}$ is singular and the exceptional loci of $X_\sigma^+$ and $\tX$ are no longer flat over $D_{r-1}(\sigma)$. In this case, the spherical twist $T_{S_k}$ in our theorem is a good candidate for the higher rank generalization of twist $T_{\sO_{\PP^1}(-1)}$ from the Atiyah flop case. 




\subsection{Applications to $\Theta$-flops} \label{sec:Theta}
In the situation of a $\Theta$-flop $X_+=C^{(g-1)} \dasharrow X_-=C^{(g-1)}$ of diagram (\ref{diag:AJ}), i.e. the case $n=0$ of \S \ref{sec:SymC}, 
Thm. \ref{thm:flop} and Lem. \ref{lem:spherical_stratified} imply: 

\begin{corollary}[$\Theta$-flops]\label{cor:Theta}
 \begin{enumerate}[leftmargin=*]
 	\item[(i)] For each $k \in \ZZ$, the functor 
 	$$\Phi_k (\blank):=  \RR {q_{2*}} (\LL q_1^*(\blank) \otimes \sO_{\widehat{X}}(k,k))  \colon D(C^{(g-1)}) \to D(C^{(g-1)})$$
is an equivalence of categories. Here $q_1,q_2$ are maps in the fibered diagram (\ref{diag:AJ}).
	 \item[(ii)] For each $k \in \ZZ$, the following functor is spherical:
	 $$S_k(\blank) = \LL (AJ)^{*} (\blank) \otimes \sO_{X_{+}}(k) \colon D(\Pic^{g-1}(C)) \to D(C^{(g-1)}).$$
In particular,  the functors $S_0=\LL (AJ)^*$ and $\LL (AJ^\vee)^*$ are spherical, with cotwist functor both given by $\otimes \sO(\Theta)^\vee \colon D(\Pic^{g-1}(C)) \simeq D(\Pic^{g-1}(C))$, where $\Theta \subset \Pic^{g-1}(C)$ is the theta divisor. The equivalences $\{\Phi_k\}_{k \in \ZZ}$ are connected by spherical twists:
	$$\Phi_{k-1}^{-1} \circ \Phi_{k} = T^{-1}_{S_{-k}}[2] \quad \text{and} \quad \Phi_{k}^{-1} \circ \Phi_{k-1} = T_{S_{-k}}[-2].$$
	 \item[(iii)] ``Flop--flop=twist" holds for $\Theta$-flops:
	$\RR q_{1*} \, \LL q_2^{*} \,\, \RR q_{2*} \, \LL q_1^{*} (\blank)  = T_{\LL (AJ)^*}^{-1}(\blank) \otimes \sO(\Theta)^\vee[2]. 
$
\end{enumerate}
\end{corollary}

\appendix
\newpage
\section{``Chess game'' method in the rectangular case} \label{sec:app}
We review the method and results on ``chess game'' \cite{JLX17} in the special case when all decompositions are rectangular. The ``chess game'' was introduced by Richard Thomas in his reinterpretation \cite{RT15HPD} of Kuznetsov's work \cite{Kuz07HPD}. It is a systematic way to compare subcategories based on the techniques of mutations \cite{B, BK}. If two subcategories $\sD_1$ and $\sD_2$ (inside an ambient category) have the property that their difference is {\em linear} in a suitable sense, then ``chess game'' method systematically compares $\sD_1$ and $\sD_2$ through analyzing the patterns of vanishings and mutations in a $2$-dimensional ``chessboard'' like Figure \ref{fig_auto}. 

\subsubsection*{Set-up for rectangular chess game} Let $S$ be a smooth scheme. Assume $\shH$ is a projective $S$-schemes, and $X, Y$ are smooth projective $S$-schemes. Let $\sO_X(1)$ and $\sO_Y(1)$ be line bundles on $X$ and $Y$, such that  there are semiorthogonal decompositions of the form:
		$$D(X) = \langle \shA, \shA(1), \ldots, \shA(e-1) \rangle, \qquad D(Y) = \langle  \shC, \shC(1), \ldots, \shC(f-1) \rangle,$$
	where $e,f \ge 2$ are integers, $\shA \subset D(X)$ and $\shC \subset D(Y)$ are $S$-linear admissible subcategories, and $\shA(k)$ (resp. $\shC(k)$) denotes the image of $\shA$ (resp. $\shC$) under the autoequivalence $\otimes \sO_X(k)$ (resp. $\otimes \sO_Y(k)$). By \cite[Thm. 5.8]{Kuz11Bas}, there is an $S$-linear semiorthogonal decomposition
	$$D(X \times_S Y) = \big \langle (\shA \boxtimes_S \shC) (\alpha,\beta) \big \rangle_{0 \le \alpha \le e-1, 0 \le \beta \le f-1}.$$
The first condition for the ``chess game" is: 
\begin{enumerate}	
	\item[\textbf{(A-1)}] Assume there is an $S$-linear functor $F: D(X \times_S Y) \to D(\shH)$, such that
	$$C_F : = \cone(1 \to RF)[-1] \simeq  \otimes \sO(-1,-1)[*] \colon D(X\times_S Y) \to D(X \times_S Y),$$
where $R$ is the right adjoint of $F$, $\sO(1,1): = \sO_X(1) \boxtimes \sO_Y(1)$, and $\sO(-1,-1)$ is the dual line bundle of $\sO(1,1)$, $[*]$ is the degree shift by some integer $* \in \ZZ$. 
\end{enumerate}

\medskip The condition \textbf{(A-$1$)} in particular holds for $F =\LL i_{\shH}^*$, where $i_{\shH} \colon \shH \hookrightarrow X \times_S Y$ is the inclusion of a divisor of line bundle $\sO(1,1)$, for example, the case $\iota \colon \shH \hookrightarrow  \PP(\sE) \times_X \PP(\sF^\vee) $ considered in the main body of this paper. In this case, the shift $[*] = [0]$ is trivial. The property  \textbf{(A-$1$)} allows us to compute the $\Hom$s on $\shH$ in terms of $\Hom$s on $X \times_S Y$:
    \begin{align*}
    \begin{split}
     R\Hom_{\shH}(F(A_1\boxtimes\, &B_1), F(A_2 \boxtimes \,B_2))  =  \cone\big(R\Hom_{X}(A_1,A_2(-1))\otimes \\
    & R\Hom_{Y}(B_1,B_2(-1)) [*] \to R\Hom_{X}(A_1,A_2)\otimes R\Hom_{Y}(B_1,B_2) \big),
    \end{split}
    \end{align*}
for any $A_1,A_2 \in D(X)$, $B_1, B_2 \in D(Y)$, where $[*]$ is the same degree shift in $C_F \simeq \sO(-1,-1)[*]$. From this we can directly show that $F|_{(\shA \boxtimes \shC)(\alpha,\beta)}: (\shA \boxtimes \shC)(\alpha,\beta) \to D(\shH)$ is fully faithful, for all $\alpha, \beta \in \ZZ$. Denote the image by:
	$$\shE(\alpha,\beta) : = F((\shA \boxtimes_S \shC) (\alpha,\beta)) \subset D(\shH), \qquad \alpha \in \ZZ, \beta \in \ZZ,$$
	$$\shE(*, \beta) = \langle \shE(k, \beta) \mid k \in \ZZ \rangle \subset D(\shH), \qquad \shE(\alpha, *) = \langle \shE(\alpha, k) \mid k \in \ZZ \rangle \subset D(\shH),$$

\begin{enumerate}
	\item[\textbf{(A-2)}] Assume that there are semiorthogonal decompositions:
	\begin{align*}
	D(\shH)  = \big \langle \sD_1, ~ \shE(*,1), \ldots,\shE(*, f-1) \big \rangle 
	 = \big \langle \sD_2, ~\shE(1,*), \dots, \shE(e-1,*)\big \rangle.
	\end{align*}
	where $\sD_1$ and $\sD_2$ are admissible subcategories of $D(\shH)$. Furthermore, assume that the decomposition can be extended to ``negative directions":		\begin{align*} 
	D(\shH) =& \big\langle \shE(*,k+1-f) , \ldots, \shE(*,0), 
	 ~ \sD_1~, \shE(*,1), \ldots, \shE(*, k-1) \big \rangle, \qquad k=1,2,\ldots, f-1.\\
	 & \big \langle \shE(k+1-e,*) , \ldots, \shE(0,*), 
	 ~ \sD_2~, \shE(1,*), \ldots, \shE(k-1,*) \big \rangle, \qquad k=1,2,\ldots, e-1.
	\end{align*}
\end{enumerate}

This is the most important condition of chess game, which says the difference of the two subcategories $\sD_1$ and $\sD_2$ inside $D(\shH)$ is ``{\em linear}".  The conditions \textbf{(A-$1$)} and \textbf{(A-$2$)}, which are satisfied by the case $\iota \colon \shH \hookrightarrow  \PP(\sE) \times_X \PP(\sF^\vee)$, and $\sD_1 = D(\PP(\sC_\sigma))$, $\sD_2 = D(\PP(\sC_{\sigma^\vee}))$ considered in this paper, is stronger than the condition in {\cite{JLX17}}. 
	
\begin{lemma}[{\cite{JLX17}}]\label{lem:app:cone} Denote by $I_i \colon \sD_i \to D(\shH)$ the inclusion functor, $i=1,2$, and denote by $I_i^*$ (resp. $I_i^!$) the left (resp. right) adjoint of $I_i$.
\begin{enumerate}
 	\item \label{lem:app:cone-1}
	If $b \in \shE(m,0)$, $m \in \ZZ$, then $\cone(b \to I_1 I_1^*(b))$ belongs to: 
 	$$\langle \shE(\alpha,\beta) \rangle_{1 \le \beta \le f-1,~ m - f + \beta \le \alpha \le m-1}.$$
	 \item  \label{lem:app:cone-2}
	  If $b \in \shE(m,1)$, $m \in \ZZ$, then $\cone(I_1 I_1^! (b) \to b)$ belongs to:
 	 $$\langle \shE(\alpha,\beta) \rangle_{1 \le \alpha \le f-1,~ m - f + \alpha \le \beta \le m-1}$$
 	\item \label{lem:app:cone-3}
	If $b \in \shE(0,m)$, $m \in \ZZ$, then $\cone(b \to I_2 I_2^*(b))$ belongs to: 
 	 $$\langle \shE(\alpha,\beta) \rangle_{1 \le \alpha \le e-1,~ m - e + \alpha \le \beta \le m-1}$$
	 \item  \label{lem:app:cone-4} If $b \in \shE(1,m)$, $m \in \ZZ$, then $\cone(I_2 I_2^! (b) \to b)$ belongs to: 
 	 $$\langle \shE(\alpha,\beta) \rangle_{1 \le \beta \le e-1,~ m - e + \beta \le \alpha  \le m-1}.$$
\end{enumerate}
\end{lemma}

\begin{proof} \eqref{lem:app:cone-1} is \cite[Lemma 3.15]{JLX17}, and \eqref{lem:app:cone-3} is \cite[Lemma 3.17]{JLX17}. The other two cases are proved similarly, once we replace the left mutations by the right mutations.
\end{proof}

The case \eqref{lem:app:cone-2} is illustrated in Figure \ref{fig_auto} (in the case $e=f=r$ and $m=1$). Other cases can be illustrated in a similar manner. 

\begin{theorem}[\cite{JLX17}]\label{thm:CG}
\begin{enumerate}
	\item If $e \ge f$, then the funtors $I_1^* I_2 \colon \sD_2 \to \sD_1$ and $I_1^*|_{\shE(m,0)} \colon \shE(m,0) \to \sD_1$, $m \in \ZZ$ are fully faithful, and induce a semiorthogonal decomposition 
		$$\sD_1 = \big \langle I_1^* I_2 \, \sD_2, ~ I_1^*\,\shE(1,0), \ldots, I_1^*\,\shE(e-f,0) \big \rangle.$$
	\item If $e \le f$, then the funtors $I_2^* I_1 \colon \sD_1 \to \sD_2$ and $I_2^*|_{\shE(0,m)} \colon \shE(0,m) \to \sD_2$, $m \in \ZZ$ are fully faithful, and induce a semiorthogonal decomposition 
		$$\sD_2 = \big \langle I_2^* I_1 \, \sD_1, ~ I_1^*\,\shE(0,1), \ldots, I_1^*\,\shE(0,f-e) \big \rangle.$$
\end{enumerate}
\end{theorem}

\begin{proof} This is \cite[Theorem 3.12]{JLX17} on ``chess game" in the rectangular case, with $\shA_0 = \ldots = \shA_{i-1} = \shA$, $\shC_0 = \ldots = \shC_{\ell-1} = \shC$, $i=e$, $\ell =f$. Note also that in this case, the argument using ``chess game'' in \cite{RT15HPD} could also be applied directly to this situation.
\end{proof}

We could summarise what we have shown in the proof of Thm. \ref{thm:flop} in terms of  ``chess game" as follows. In the above situation, assume \textbf{(A-$1$)} and \textbf{(A-$2$)} hold.

\begin{theorem}\label{thm:CG:auto} If $e=f$, and the ``chess game" is {\em homogenous}, i.e. there is an autoequivalence $\sigma \in \operatorname{Auteq}(D(\shH))$, such that $\sigma (\shE(\alpha,\beta)) = \shE(\alpha+1,\beta+1)$ for all $\alpha, \beta \in \ZZ$. Moreover, assume the rotation functor $\foR: = I_1^* \circ \sigma \circ I_1 \colon \sD_1 \to \sD_1$ is an autoequivalence. Then
	\begin{enumerate} 
	\item For each $k \in \ZZ$, the functor 
	$\Psi_k := I_2^* \circ \sigma^k \circ I_1 \colon \sD_1 \to \sD_2$
is an equivalence.
	\item For each $k \in \ZZ$, denote $E_k \colon \shE (1-k,1) \hookrightarrow D(\shH)$ the inclusion functor, and let $J_k: = I_1^! \circ E_k \colon \shE (1-k,1) \to \sD_1$ be the composition. Then the twist functor $T_{J_k} \colon \sD_1 \to \sD_1$ is an autoequivalence. Moreover, for each $k$, the following holds:
	$$\Psi_{k-1}^{-1} \circ \Psi_{k} = T^{-1}_{J_{k}} \circ \foR, \qquad \Psi_{k}^{-1} \circ \Psi_{k-1} = \foR^{-1} \circ T_{J_{k}}.$$
	\end{enumerate}
\end{theorem}

\begin{example}[Projectivization] In the main situation considered in this paper, $F = \LL \iota_{\shH}^*$, where $\iota_{\shH}  \colon \shH \hookrightarrow  \PP(\sE) \times_X \PP(\sF^\vee)$ is an inclusion of an $\sO(1,1)$-divisor, $I_1 = \RR j_{1*} \LL p_1^* \colon \sD_1 = D(\PP(\sC_\sigma)) \hookrightarrow D(\shH)$, $I_2 = \RR j_{2*} \LL p_2^* \colon  \sD_2 = D(\PP(\sC_{\sigma^\vee})) \hookrightarrow D(\shH)$. Then $C_F = \cone(1 \to \RR \iota_{\shH *} \LL \iota_{\shH}^*)[-1] = \otimes \sO_{\shH}(-1,-1)$, $\sigma = \otimes \sO_{\shH}(1,1)$. The rotation functor is $\foR = [2]$ by Lem. \ref{lem:Rot:Hyp}. In the proof of Thm. \ref{thm:duality}, we showed that \textbf{(A-$1$)} and \textbf{(A-$2$)} are satisfied. In the case $e=f$, the condition of Thm. \ref{thm:CG:auto} is also satisfied, and Thm. \ref{thm:CG:auto} reduces to Thm. \ref{thm:flop}.
\end{example}

\begin{example}[Standard flips] \label{ex:standardflips} Let $S$ and $X_\pm$ be smooth varieties. Assume $\sF$ and $\sE$ are vector bundles over $S$ with ranks $f$ and $e$. Let $i_{1} \colon P_+ = \PP_S(\sE) \hookrightarrow X_+$ and $i_2 \colon P_- = \PP_S(\sF^\vee) \hookrightarrow X_-$ be regular closed immersions such that $\sN_{i_1} = \sO_{P_+}(-1) \otimes \sF$, and $\sN_{i_2} = \sO_{P_-}(1) \otimes \sE^\vee$. Assume furthermore that $\widetilde{X} =\Bl_{P_+} X_+ = \Bl_{P_-} X_-$ is the common blowup, with inclusion of exceptional divisor $i_E \colon E \to \widetilde{X}$. Then $E =P_+ \times_S P_- = \PP_S(\sE) \times_S \PP_S(\sF^\vee)$, and $\sO_{\widetilde{X}}(E)|_E = \sO_E(-1,-1)$. Hence there is a commutative diagram, with maps as indicated:
\begin{equation*}
\begin{tikzcd}[row sep=1 em, column sep=2 em]
	 & & E = P_+ \times_S P_- \ar[hook]{d}[swap]{i_E} \ar{lldd}[swap]{p_1} \ar{rrdd}{p_2}& & \\
	 & & \widetilde{X} \ar{ld}[swap]{q_1} \ar{rd}{q_2} & & \\
	P_+ \ar[hook]{r}{i_1}&X_+ & &X_- &P_- \ar[hook']{l}[swap]{i_2}
\end{tikzcd}
\end{equation*}
In the rest of the example, we assume all functors are derived. Therefore $F = i_{E*} \colon D(E) = D(P_+ \times_S P_-) \to D(\widetilde{X})$ satisfies
 	$$C_F = \cone(1 \to i_E^! i_{E*}) [-1] = \otimes \sO_E(E)[-2] = \sO_E(-1,-1)[-2].$$
Hence condition \textbf{(A-$1$)} is satisfied by $F$ and $\shH = \widetilde{X}$. By Orlov's blowup formula:
 	\begin{align*} 
		D(\tilde{X})	
		& = \langle  q_1^* D(X_+), i_{E*} (D(P_+)\boxtimes_S \sO_{P_-}), i_{E*} (D(P_+)\boxtimes_S \sO_{P_-}(1)), \ldots,  i_{E*} (D(P) \boxtimes_S \sO_{P_-}(f-1) \rangle); \\		
		& = \langle  q_2^* D(X_-), i_{E*}  (\sO_{P_+} \boxtimes_S D(P_-) ),  i_{E*} (\sO_{P_+}(1) \boxtimes_S D(P_-) ), \ldots,  i_{E*}  (\sO_{P_+}(e-1) \boxtimes_S D(P_-) )\rangle,
	\end{align*}
which could also be extended to ``negative directions", see Thm. \ref{thm:blow-up}. Denote by $\shE(\alpha,\beta)$ the image of $D(S)(\alpha-1,\beta-1)\subseteq D(P_+ \times_S P_-)$ under the functor $F = i_{E*}$, and $I_1 = q_1^*$, $I_2 = q_2^*$. Hence \textbf{(A-$2$)} is also satisfied. Assume $e \ge f$, then Thm. \ref{thm:CG} implies:
	\begin{align*}
	D(X_+) =  \big\langle q_{1*} \, q_2^* D(X_-), i_{1*} (\psi_1^*D(S)), \ldots,   i_{1*} (\psi_1^*D(S) \otimes \sO_{P_+}(e-f-1)) \big \rangle,
	\end{align*}
where $\psi_1 \colon P_+ \to S$ is the projection. Set $\sigma := \otimes \sO_{\widetilde{X}}(-E) \colon D(\widetilde{X}) \to D(\widetilde{X})$. By Lem. \ref{lem:mut:Bl}, the rotation functor for blowup is $\foR = \Id$. Hence if $e=f$, all conditions of Thm. \ref{thm:CG:auto} are satisfied. Therefore for each $k \in \ZZ$, the functor
 	$$\mu_k := I_2^* \circ \sigma^k \circ I_1 
	= q_{2*}(q_1^*(
	\blank) \otimes \sO_{\widetilde{X}}((e-k)E)): D(X_+) \to D(X_-)$$
is an equivalence of categories. The inverse of $\mu_k$ is given by 
	$$\mu_k^{-1} = q_{1*} (q_2^*(\blank) \otimes \sO_{\widetilde{X}}(kE)) \colon D(X_-) \to D(X_+).$$
Furthermore, denote by $E_k$ the inclusion functor of $\shE(1-k,1)= i_{E*}(D(S)(-k,0)) \subseteq D(\widetilde{X})$. Then $J_k = I_1^! E_k = q_{1*} E_k$ is the functor 
	$$J_{k}  =  i_{1*} (\psi_1^*(\blank) \otimes \sO_{P_+}(-k)) \colon D(S) \to D(X_+).$$ 
Denote $T_{J_{k}}$ the twist functor around $J_k$. Since $\foR = \Id$, Thm. \ref{thm:CG:auto} further implies:
	$$ \mu_{k}^{-1} \circ \mu_{k-1} =  T_{J_{k}}, \quad \text{and} \quad \mu_{k-1}^{-1} \circ \mu_{k} =  T^{-1}_{J_{k}}.$$
This agrees with the result of Addington, Donovan and Meachan \cite[Thm. A]{ADM}. In fact, our functors $\mu_k$'s are related with their functors $\mathrm{BO}_k$'s by $\mu_k = \mathrm{BO}_{e-k}$ and $\mu_k^{-1} = \mathrm{BO'}_{k}$,  $e = n+1$. The functor $J_k$ is EZ-spherical in the sense of Horja \cite{Hor}; see \cite{ADM} for more details.
\end{example}

\subsection{Atiyah flops} \label{sec:app:Atiyah} In this subsection, we provide a direct proof of the equality mentioned in Ex. \ref{ex:Atiyah} for the Atiyah flop $\pX \dasharrow \tX$ of \S \ref{sec:springer}. Let $X = \End(V)$, where $V = \CC^{\oplus 2}$, and let $\sigma \colon V\otimes \sO_X \to V \otimes \sO_X$ be the tautological map. Then $X_{\sigma} \subset X = \End(V)$ is the determinantal hypersurface defined by $\det \sigma \colon \sO_X \to \sO_X$. Furthermore, $\pX \subset X \times P_+$ and $\tX \subset X \times P_-$ are two small crepant resolutions of the threefold $X_\sigma$, where $P_+ = \PP(V) \simeq \PP^1$, $P_- = \PP(V^\vee) \simeq (\PP^1)^*$. Hence we have a commutative diagram:
$$
	\begin{tikzcd}
		& P_+  \ar{d}{p} \ar[hook]{r}{i}& \pX \ar[hook]{r}{\iota} \ar{rd}{\pi} & X \times P_+ \ar{d}{q}\\
		& \Spec \kk \ar[hook]{rr}{j} & & X =\End(V).
	\end{tikzcd}
	$$
Here we use $i$, $\pi$, etc, instead of $i_1, \pi_+$, etc, for simplicity of notations. The maps $i$ and $\iota$ are closed immersions given by regular sections of vector bundles, and their normal bundles are given by $\sN_i = \sO_{P_+}(-1) \otimes V^\vee$, $\sN_\iota = \sO_{P_+}(1)\otimes V^\vee|_{\pX}$. 


\begin{lemma}\label{lem:app:Atiyah} For any $A \in D(\pX)$, $T_{S}(A) = T_{i_*\sO_{\PP_+}(-1)}(A) [2]$, where $S = \LL \pi^* \colon D(X) \to D(\pX)$.
\end{lemma}
	
\begin{proof} For simplicity of notations, all functors in this proof are assumed to be {\em derived}.

	\medskip
	{\em (Proof via ``chess game".)} Let $\widetilde{X} = \Bl_{P_+} \pX = \Bl_{P_-} \tX$ be the blowup of $\pX$ along $P_+$ (resp. $\tX$ along $P_-$), and denote $i_E \colon E \to \widetilde{X}$ the inclusion of the exceptional divisor. Let $\shH = \Bl_{\pX}(X \times P_+) = \Bl_{\tX} (X \times P_-)$ be the blowup of $X \times P_+$ along $\pX$ (resp. $X \times P_-$ along $\tX$). Then there are semiorthogonal decompositions:
	\begin{align}
	 D(\widetilde{X}) & = \langle q_1^*(D(\pX)), i_{E*} (D(P_+) \boxtimes \sO_{P_-})\rangle = \langle q_2^* (D(\tX)), i_{E*}(\sO_{P_+} \boxtimes D(P_-))\rangle;    \label{eq:appA:tX}\\
	  D(\shH)  &=  \langle I_1 (D(\pX)), D(P_+) \boxtimes D(X) (0,1)\rangle = \langle I_2(D(\tX)), D(P_-) \boxtimes D(X) (1,0) \rangle. \label{eq:appA:H}
	\end{align} 
The lemma follows from computing the same functor 
	$\Phi = q_{1*} \, q_2^{*} \,q_{2*} \, q_1^{*} $
twice.
The chess game for  \eqref{eq:appA:tX} implies that $\Phi = \mu_0^{-1} \circ \mu_{1}= T_{i_*\sO_{\PP_+}(-1)}^{-1}$, see Ex. \ref{ex:standardflips}.
The chess game for \eqref{eq:appA:H} implies $\Phi = \Phi_0^{-1} \circ \Phi_{1} = T_S^{-1}[2]$, see Thm. \ref{thm:flop}. Hence the lemma follows.
	
	\medskip
	{\em (Proof via direct computations.)} Denote $F = i_*(p^* (\blank) \otimes \sO_{P_+}(-1))$, $F^R = p_*(i^!(\blank) \otimes \sO_{P_+}(1)) = p_*(i^*(\blank) \otimes  \sO_{P_+}(-1))[-2]$, $S =\pi^*$, $S^R = \pi_*$. There are exact triangles:
		$$F F^R (A) \to A \to T_{i_*\sO_{\PP_+}(-1)}(A) \xrightarrow{[1]}, \qquad S S^R(A) \to A \to T_S(A)  \xrightarrow{[1]}.$$
		
	(i) If $A = \sO_{\pX}$. Then $F^R (A)=0$, hence $T_{i_*\sO_{\PP_+}(-1)}( \sO_{\pX}) =  \sO_{\pX}$. On the other hand, $S^R (A) = [\sO_X \xrightarrow{\det \sigma} \sO_X]$, $S S^R(A) =  \sO_{\pX} \oplus  \sO_{\pX}[1]$. Therefore $T_S( \sO_{\pX}) =  \sO_{\pX}[2]$.
	
	(ii) If $A = \sO_{\pX}(1)$, where $ \sO_{\pX}(k) : = \iota^* (\sO_X \boxtimes \sO_{P_+}(k))$ for $k \in \ZZ$. Then $F^R (A)= H^0(P_+, \sO_{P_+}) = \CC$, $FF^R = i_*(\sO_{P_+}(-1))[-2]$. By using the Koszul resolution $i_* \sO_{P_+} \simeq [\wedge^2 V \otimes \sO_{\pX}(2) \to V \otimes  \sO_{\pX}(1) \to  \sO_{\pX}]$,  we obtain $T_{i_*\sO_{\PP_+}(-1)}( \sO_{\pX}(1) ) = [V \otimes \sO_{\pX} \to  \sO_{\pX}(-1)][-1]$ (where the degree shift $[-1]$ means $V$ is at degree $0$ and $\sO_{\pX}(-1)$ is at degree $1$). On the other hand, $S^R (A)= [V \xrightarrow{\sigma} V]$, and $S S^R(A) =  [V \otimes \sO_{\pX}\xrightarrow{\pi^* \sigma} V \otimes \sO_{\pX}]$. By definition, $\pX$ is the locus where the composition $V \xrightarrow{q^* \sigma} V \twoheadrightarrow \sO_{P_+}(1)$ is zero, hence $\pi^*\sigma$ factorises through $V  \to \sO_{P_+}(-1)|_{\pX} \subseteq  V$. By using the Euler sequence $\sO_{\pX}(-1) \to V \otimes \sO_{\pX} \to \sO_{\pX}(1)$, we obtain $T_S( \sO_{\pX}(1)) = \cone(SS^R(A) \to A) = [V \otimes \sO_{\pX} \to \sO_{\pX}(-1)][1] = T_{i_*\sO_{\PP_+}(-1)}( \sO_{\pX}(1))[2]$.
	
Since $D(\pX)$ is spanned by $\sO_{\pX}$ and $\sO_{\pX}(1)$, the lemma is proved by (i) and (ii). For illustration, we also compute the following case:

	(iii) If $A = i_*\sO_{\PP_+}(-1)$ (the spherical object). Then $T_A (A) = A[-2]$ by $F^R F = \Id \oplus [-3]$. (This is the first property of $A$ for being $3$-spherical, see \cite[Ex. 8.5(ii)]{Huy}). On the other hand, $S^R (A) = \pi_* i_* \sO_{\PP_+}(-1) =  j_* p_* \sO_{\PP_+}(-1) = 0$. Therefore $T_S (A) = A = T_A (A)[2]$.
\end{proof}

\section{Dimensions, smoothness and Brill--Noether loci} \label{sec:app:SymC}
In this section we assume $\kk = \CC$. For a smooth complex projective curve $C$, $d,k \in \ZZ$, the {\em Brill--Noether loci} is defined as:
	$$W_{d}^k := W_{d}^k(C) : = \{\sL \mid \dim H^0(C,\sL) \ge k+1\} \subset \Pic^d(C).$$
The expected dimension of $W_d^k$ is the {\em Brill--Noether number} $\rho(g,k,d) := g - (k+1)(g-d+k)$. From the works of Kempf \cite{Kempf}, Kleiman--Laksov \cite{KL1,KL2} and Fulton--Lazarsfeld \cite{FL}, we know that
$W_{d}^k \ne \emptyset$ if $\rho(g,k,d) \ge 0$, $W_{d}^k$ is connected if $\rho(g,k,d) \ge 1$. The Brill--Noether inequality and {\em Martens theorem} (see \cite[IV, \S5]{ACGH}) states for $g \ge 3$, the dimension of $W_d^k$ satisfies:
	$$ \rho(g,d,k) \le \dim W_d^k \le d - 2k,$$
if $(d,k) \in \{1 \le d \le g-1, 0 \le k \le \frac{d}{2}\} \cup \{g-1 \le d \le 2g-3, d-g +1 \le k \le \frac{d}{2}\}$. \footnote{The above result is only stated for the range $1 \le k \le \frac{d}{2}$ and $2 \le d \le g-1$ in \cite[IV, \S5]{ACGH}. Notice that if $k=0$, then $\dim W_{d}^0 = \rho(g,d,0) = d$ holds; If $d - 2k <0$, then $W_d^k=\emptyset$ by Clifford inequality. Then all other cases follow from above cases by the canonical isomorphisms $W^k_d \simeq W^{g-d+k-1}_{2g-2-d}$.}

For a smooth projective curve $C$ of genus $g \ge 1$ and integer $n \ge 0$, the symmetric products $C^{(g-1\pm n)}$ are equipped with Abel--Jacobi maps to $\Pic^{g-1+n}(C)$ as in \S \ref{sec:SymC}. Consider:
	\begin{align*}
	\widehat{X} : & = C^{(g-1+n)} \times_{\Pic^{g-1+n}(C)} C^{(g-1-n)} \\
		& =   \{(D,D') \mid D \in C^{(g-1+n)}, D' \in C^{(g-1-n)}, D+D' \equiv K_C \} 
	\end{align*}
Let $d: = g-1+n$, and denote $\widehat{\pi} \colon \widehat{X} \to \Pic^d(C)$ the natural projection. Notice that $\widehat{X} \ne \emptyset$ if and only if $0 \le n \le g-1$. 
\begin{lemma} \label{lem:app:exp.dim} If $\widehat{X} \ne \emptyset$, then it has the expected dimension 
	$$\dim \widehat{X} = g-1.$$
\end{lemma}
\begin{proof}  
By assumption $0 \le n \le g-1$. 
If $n = g-1$ then the result clearly holds since $AJ^\vee \colon C^{(g-1-n)} \simeq [\omega_C] \in \Pic^{2g-2}(C)$, and $\widehat{X} = \PP_{\rm sub}(H^0(C,\omega_C))\simeq \PP^{g-1} \subset C^{(2g-2)}$. Hence we need only consider the case $0 \le n \le g-2$ and $g \ge 2$. If $g=2$ then $n=0$, $\widehat{X} \simeq C$ clearly satisfies the condition. Now we focus on the case $g \ge 3$, $d = g-1+n \in [g-1, 2g-3]$. Notice that $\widehat{X}$ is stratified by the locally closed subschemes $\widehat{\pi}^{-1}(W^k_d \,\backslash\, W^{k+1}_d)$, $n \le k \le \frac{d}{2}$. For every irreducible component $Z \subset W^k_d \,\backslash\, W^{k+1}_d$, $\widehat{\pi}^{-1}(Z)$ is a $\PP^k \times \PP^{k-n}$-bundle over $Z$, therefore it follows from Martens theorem above that:
	$$\dim \widehat{\pi}^{-1}(Z) \le d - 2k + k + k-n = g-1.$$
However for $k=n$, $C^{(g-1-n)}$ maps birationally onto $W^{n}_d \simeq W^{0}_{g-1-n}$. Since $W^{n+1}_d \simeq W^{1}_{g-1-n} \ne W^{0}_{g-1-n} \simeq W^{n}_d$,  and the restriction of $AJ^\vee \colon C^{(g-1-n)} \to \Pic^{g-1+n}(C)$ over the non-empty stratum $W^{n}_d \,\backslash\, W^{n+1}_d$ is a bijection, therefore $\dim (W^{n}_d \,\backslash\, W^{n+1}_d) = g-1-n$. Hence
 	$$\dim \widehat{\pi}^{-1}(W^{n}_d \,\backslash\, W^{n+1}_d) = g-1.$$
Therefore the expected dimension $g-1$ is achieved by $\widehat{X}$. \end{proof}

We next focus on the smoothness of $\widehat{X}$. 
Recall $C$ is called a {\em Petri curve} if the Petri map:
	$$\mu_\sL \colon H^0(C,\sL) \otimes H^0(C, \omega_C \otimes \sL^\vee) \to H^0(C,\omega_C)$$
is injective for any $\sL \in \Pic(C)$. By the work of Gieseker \cite{Gie} (also Griffiths--Harris \cite{GH}, Eisenbud--Harris \cite{EH} and Lazarsfeld \cite{Laz}; see also \cite[V]{ACGH}), Petri curves form a open dense subset of the moduli space $\shM_g$ of genus $g$ curves.
\begin{lemma} For a general curve $C$ (in the sense of Petri), $\widehat{X}$ is smooth.
\end{lemma}

\begin{proof} Since for a Petri curve $C$, and any $d$ and $k$, by \cite[Prop. IV (4.2), Thm. V (1.7)]{ACGH}, $W^k_d$ (is either empty or) has the expected dimension $\rho(g,d,k)$, and $W^{k}_d \, \backslash \, W^{k+1}_d$ is smooth. Hence the claim follows from the following general result.
\end{proof}

\begin{lemma} Let $X$ be a smooth complex variety, and $\sG$ be a coherent sheaf of homological dimension $1$ and rank $r$. Assume that for all $i \ge 0$, $X^{>r+i-1}(\sG) \, \backslash \, X^{>r+i}(\sG)$ is (either empty or) a smooth subvariety of the expected codimension $i(i+r)$. Then $\PP(\sG)$, $\PP(\sExt^1(\sG,\sO))$ and $\PP(\sG) \times_X \PP(\sExt^1(\sG,\sO))$ are smooth of expected dimensions (see Thm. \ref{cor:projectivization}).
\end{lemma}

\begin{proof} We only show that $\widehat{X}: = \PP(\sG) \times_X \PP(\sExt^1(\sG,\sO))$ is smooth and has the expected dimension; The other two cases are similar (and simpler). Denote $\widehat{\pi} \colon \widehat{X} \to X$ the projection.

First, we consider the ``universal" space of Homs $H = \Hom(W,V)$ between vector spaces $W$ and $V$, $\dim W = m$, $\dim V = n$. Let $\tau \colon W \otimes \sO_H \to V \otimes \sO_H$ be the tautological map, i.e. $\tau|_A = A  \colon W \to V$, for any $A = (a_{ij}) \in \Hom(W,V)$. For any $0 \le \ell \le \min\{m,n\}$, let $\DD_\ell \subseteq H$ be the degeneracy locus of $\tau$. For any $A_0 \in \DD: = \DD_\ell \backslash \DD_{\ell-1}$, up to ${\rm GL}(W) \times {\rm GL}(V)$-action, we may assume 
	$A_0 = \big(\begin{smallmatrix}
	1_\ell & 0  \\
	0 & 0  \\
	\end{smallmatrix} \big) \colon W \to V$.  Set $K = \Ker A_0$, $C = \Coker A_0$.
Consider the open neighbourhood $U = \{ A = (a_{ij}) \mid a_{ii} \ne 0, i=1,\ldots, \ell\}\subset  \Hom(W,V)$ of $A_0$. Then standard row and column operations show that there is an identification:
	$$\phi \colon U \xrightarrow{\sim} H' \times E, \qquad \text{where} \qquad H' = \Hom(K,C), \quad E = \AA^{(m-\ell)\ell + (n-\ell)\ell} \times (\GG_{m})^{\times \ell^2},$$
 such that $\DD_\ell \cap U = \phi^{-1} ( \{0 \} \times E)$, see for example \cite[Prop. 2.4]{BV}. Notice that $H'=\Hom(K,C) = N_{\DD/H}|_{A_0}$ is the {\em normal} space of $\DD$ to $H$ at $A_0$.

Denote $\tau$ resp. $\tau'$ the tautological on $H$ resp. $H' : = \Hom(K,C)$, and set $\widehat{H} =\PP_H({\rm Coker}(\tau)) \times_H \PP_H({\rm Coker}(\tau^\vee))$ and $\widehat{H}' = \PP_{H'}({\rm Coker}(\tau')) \times_{H'} \PP_{H'}({\rm Coker}(\tau'^\vee))$. Then 
	$$\widehat{H}' = \Tot_{\PP_- \times \PP_+}(\Omega_{\PP_-}(1) \boxtimes \Omega_{\PP_+}(1)) \subseteq H' \times \PP_- \times \PP_+,$$
where $\PP_- := \PP_\CC(K)$ and $\PP_+ := \PP_\CC(C)$.  The fiber over of $\widehat{H}' \to H'$ over $\{0\} \in H'$ is the zero section $\{0\}\times \PP_- \times \PP_+$. Under the identification $\phi$, it is clear that $\widehat{H}|_U \simeq \widehat{H}' \times E$ is smooth, and there is a sequence of regular closed immersions of smooth subvarieties:
	$$
	\begin{tikzcd}
	& \widehat{H}|_{\DD \cap U} =  (\DD \cap U) \times  \PP_- \times \PP_+ \ar[hook]{r}& \widehat{H}|_U \ar[hook]{r} & U \times \PP_- \times \PP_+.
	\end{tikzcd}
	$$


Next, we consider the general case. Since the problem is local, we may assume $\sG$ is the cokernel of an injective $\sO_X$-module map $\sigma \colon \sF \to \sE$, and $\sF = W \otimes \sO_X$, $\sE = V \otimes \sO_X$, where $W,V$ are vector spaces.  For any $\ell$, and any $x \in D: = D_\ell(\sigma) \backslash D_{\ell-1}(\sigma)$, there is a map $f \colon X \to H = \Hom(W,V)$ such that $f(x) = A_0$ and $f^{-1}(\DD_\ell) = D_\ell(\sigma)$, $f^{-1}(\DD) = D$. The base-change $f$ induces a sequence of inclusions 
	$$
	\begin{tikzcd}
	& \widehat{\pi}^{-1}(D) \cap f^{-1}(U)  \ar[hook]{r}&  \widehat{X} \cap f^{-1}(U)  \ar[hook]{r} & f^{-1}(U)  \times \PP_- \times \PP_+.
	\end{tikzcd}
	$$
Denote $g = f \times \id_{\PP_-} \times \id_{\PP_+} \colon  f^{-1}(U)  \times \PP_- \times \PP_+ \to U \times \PP_- \times \PP_+$. Since $x$ is a smooth point, the composition $T_{x} X \xrightarrow{(df)_*} T_{f(x)} U \twoheadrightarrow N_{\DD/U}|_{f(x)}$ is surjective (see \cite[\S 5.1, page 54]{FP}). For any $p \in \widehat{\pi}^{-1}(x)$, denote $N_{g(p)}$ is the normal space of $\widehat{H}|_U$ to $U \times \PP_- \times \PP_+$ at $g(p)$. Since the map $N_{\widehat{H}|_{\DD \cap U}/U \times \PP_- \times \PP_+}|_{g(p)} \twoheadrightarrow N_{g(p)}$ is surjective, therefore the composition 
	$$T_{p}(X \times \PP_- \times \PP_+) \xrightarrow{(dg)_*} T_{g(p)}(U \times \PP_- \times \PP_+) \twoheadrightarrow N_{g(p)}$$
is surjective. 
Hence $g$ is transverse to $\widehat{H}|_U$ at $g(p)$. Therefore $\widehat{X}$ is smooth at $p$ of expected dimension (see \cite[\S 5.1, page 53]{FP}).
\end{proof}

\end{document}